\documentclass[a4paper,12pt,reqno]{amsart}
\usepackage{amsmath, amsfonts, amssymb}
\usepackage{enumerate}
\usepackage[hang,flushmargin]{footmisc} 
\usepackage{amsthm}
\usepackage{mathtools}
\usepackage{calc}
\usepackage{bm}
\usepackage{comment}
\usepackage{hyperref}
\makeatletter
   \providecommand\@dotsep{5}
 \makeatother

\newtheorem{theorem}{Theorem}[section]
\newtheorem{lemma}[theorem]{Lemma}

\newtheorem{proposition}[theorem]{Proposition}

\theoremstyle{definition}
\newtheorem{defn}[theorem]{Definition}
\newtheorem{remark}[theorem]{Remark}
\newtheorem{example}[theorem]{Example}
\def\E{\mathbb{E}}
\def\Z{\mathbb{Z}}
\def\R{\mathbb{R}}
\def\T{\mathbb{T}}
\def\C{\mathbb{C}}
\def\N{\mathbb{N}}

\def\Q{\mathbb{Q}}

\def\cR{\mathcal{R}}

\newcommand{\ud}{\,\mathrm{d}}
\newcommand{\id}{\mathrm{id}}

\DeclareMathOperator{\Span}{Span}

\DeclareMathOperator{\supp}{Supp}

\DeclareMathOperator{\cB}{\mathcal{B}}

\DeclareMathOperator{\cE}{\mathcal{E}}
\DeclareMathOperator{\cG}{\mathcal{G}}
\DeclareMathOperator{\cP}{\mathcal{P}}

\DeclareMathOperator{\cV}{\mathcal{V}}

\DeclareMathOperator{\hd}{\tau}
\DeclareMathOperator{\cT}{\vartheta}

\let\originalleft\left
\let\originalright\right
\renewcommand{\left}{\mathopen{}\mathclose\bgroup\originalleft}
\renewcommand{\right}{\aftergroup\egroup\originalright}

\providecommand{\norm}[1]{\left\lVert #1 \right\rVert}

\renewcommand{\subset}{\subseteq}
\renewcommand{\supset}{\supseteq}

\begin{document}
\title[On linear configurations and invariant hypergraphs]{On linear configurations in subsets of compact abelian groups, and invariant measurable hypergraphs}
\date{}

\author{Pablo Candela}
\address{Alfr\'ed R\'enyi Mathematical Research Institute\newline
	\indent Budapest, Hungary}
\email{candela83@gmail.com}

\author{Bal\'azs Szegedy}
\address{Alfr\'ed R\'enyi Mathematical Research Institute\newline
\indent Budapest, Hungary}
\email{szegedyb@gmail.com}

\author{Llu\'is Vena}
\address{Department of Mathematics\newline
\indent University of Toronto, Toronto, Canada}
\email{lluis.vena@utoronto.ca}

\subjclass{Primary 11B30, 22C05, 05C65; Secondary  22F10, 11C20}
\keywords{Linear configurations, hypergraphs, removal results, compact abelian groups}

\begin{abstract}
We prove an arithmetic removal result for all compact abelian groups, generalizing a finitary removal result of Kr\'al', Serra and the third  author. To this end, we consider  infinite measurable hypergraphs that are invariant under certain group actions, and for these hypergraphs we prove a symmetry-preserving removal lemma, which extends a finitary  result of the same name by the second author. We deduce our arithmetic removal result by applying this lemma to a specific type of invariant measurable hypergraph. As a direct application, we obtain the following generalization of Szemer\'edi's theorem: for any compact abelian group $G$, any measurable set $A\subset G$ with Haar probability $\mu(A)\geq\alpha>0$ satisfies
\[
\int_G\int_G\; 1_A\big(x\big)\; 1_A\big(x+r\big) \cdots 1_A\big(x+(k-1)r\big) \;\ud\mu(x)\ud\mu(r) \geq c,
\]
where the constant $c=c(\alpha,k)>0$ is valid uniformly for all $G$.  This result is shown to hold more generally for any translation-invariant system of $r$ linear equations given by an integer matrix with coprime $r\times r$ minors.
\end{abstract}

\maketitle

\section{Introduction}
This paper concerns the general question of the extent to which linear configurations of a given type must occur in subsets of abelian groups. Given a matrix $M\in \Z^{r\times m}$, and a subset $A$ of an abelian group $G$, we consider the set of elements $x\in A^m$ solving the system $Mx=0$, that is the set $A^m\cap \ker_G M$. In relation to the above question, it is a well-known fruitful approach to examine what can be deduced about $A$ if the set $A^m\cap \ker_G M$ occupies a small proportion of the total set of configurations $\ker_G M$. In this direction, useful information is provided by what are often called arithmetic removal results. The following example treats the case of simple abelian groups $G=\Z_p$.

\begin{theorem}\label{thm:KSVremlem}
Let $m,r$ be positive integers, with $m\geq r$. Then for any $\epsilon>0$ there exists $\delta>0$ such that the following holds. Let $M$ be a matrix of rank $r$ in $\Z^{r\times m}$ and suppose that $A_1,A_2,\ldots, A_m$ are subsets of $\Z_p$ such that $|A_1\times A_2 \times \cdots \times A_m \cap \ker_{\Z_p} M | \leq \delta |\ker_{\Z_p} M|$. Then there exist $R_1\subset A_1,\ldots, R_m\subset A_m$ such that $|R_j|\leq \epsilon p$ for every $j\in [m]$, and $\big(\prod_{j\in [m]} A_j\setminus R_j\big) \cap \ker_{\Z_p} M =\emptyset$.
\end{theorem}

As a consequence, if $|A^m\cap \ker_{\Z_p} M| \leq \delta | \ker_{\Z_p} M|$, then it is possible to eliminate all these solutions in $A^m$ by removing at most $\epsilon p$ elements from $A$. Thus $A$ must be of the form $B\cup R$, where $|R|\leq \epsilon p$ and $B$ is what we call an \emph{$M$-free set}, that is  it satisfies $B^m\cap \ker_G M=\emptyset$.

Theorem \ref{thm:KSVremlem} was proved by Shapira \cite{Shap} and independently by Kr\'al', Serra and the third author \cite{KSV2}. (Strictly speaking, the result was proved more generally for finite fields.) This result confirmed a conjecture of Green from \cite{GAR}. In that paper, Green introduced the notion of such removal results as arithmetic counterparts of well-known combinatorial removal results from graph theory, and he proved a version of Theorem \ref{thm:KSVremlem} for a single linear equation on an arbitrary finite abelian group. For more background on the relation between arithmetic and combinatorial removal results, the reader is referred to the survey \cite{C&F}, especially Section 4 therein.

One of the central consequences of Theorem \ref{thm:KSVremlem} is a  general form of Szemer\'edi's famous theorem on arithmetic progressions \cite{Sze75}, Theorem \ref{thm:StrongZpSzem} below. To state the result, we use the following terminology. We say that a matrix $M\in \Z^{r\times m}$ is \emph{invariant} if its columns sum to zero, that is if $M (1,1,\ldots,1)^T=0$; equivalently, for any abelian group $G$, the set $\ker_G M$ is invariant under translations by constant elements $(t,t,\ldots, t)$, $t\in G$. Examples of configurations given by invariant matrices include arithmetic progressions of an arbitrary fixed length.

\begin{theorem}\label{thm:StrongZpSzem}
Let $m,r$ be positive integers, with $m\geq r$. For any $\alpha>0$ there exists $c=c(\alpha,m)>0$ such that the following holds. Let $M$ be an invariant matrix of rank $r$ in $\Z^{r\times m}$, and let $A$ be a subset of $\Z_p$ of cardinality at least $\alpha p$. Then we have \[
|A^m \cap \ker_{\Z_p} M|\,/\, |\ker_{\Z_p} M|\geq c.
\]
\end{theorem}
In particular, for any positive integer $k$, the set $A$ must contain a positive proportion $c(\alpha,k)$ of the total number $p^2$ of $k$-term progressions in $\Z_p$. The deduction of Theorem \ref{thm:StrongZpSzem} from Theorem \ref{thm:KSVremlem} is very short, we record a proof in a more general context at the end of Section \ref{section:CayleyHypDefn}.

If $M$ is not invariant, then the conclusion of Theorem \ref{thm:StrongZpSzem} fails, in that there exists $\alpha=\alpha(M)>0$ such that in any group $\Z_p$ there is an $M$-free set of size at least $\alpha p$. This can be shown using a simple adaptation of the argument from \cite[Theorem 2.1]{Ruz}.

Thus for $G=\Z_p$, as a direct consequence of Theorem \ref{thm:KSVremlem}, the question recalled at the beginning of this introduction receives a strong answer (Theorem \ref{thm:StrongZpSzem})  which is also exhaustive as far as systems of linear equations are concerned.\footnote{The answer is strong in a qualitative sense.  The quantitative problem of obtaining optimal estimates for the function $c(\alpha,M)$ in Theorem \ref{thm:StrongZpSzem} is a vast and very interesting one, that includes improving the bounds for Szemer\'edi's theorem.  For the latter theorem the current best general bounds were given in \cite{GSz}; see also \cite{Bloom,G&T2,Sanders} for the latest improvements in the cases $k=3,4$.} It is natural to wonder whether this picture holds for more general abelian groups.

Given $M\in \Z^{r\times m}$ of rank $r$, let us denote by $d_r(M)$ the  determinantal divisor of $M$ of order $r$, that is the greatest common divisor of the non-zero determinants of $r\times r$ submatrices of $M$; 
see \cite[Chapter II, \S 13]{Newman}. We shall not consider determinantal divisors of lower order, and will therefore refer to $d_r(M)$ simply as `the determinantal' of $M$.

 Under the assumption that $d_r(M)=1$, Kr\'al', Serra and the third author generalized Theorem \ref{thm:KSVremlem} to all finite abelian groups, obtaining\footnote{Theorem 1 in \cite{KSVGR} actually assumes that $\gcd(d_r(M),|G|)=1$, a weaker assumption than $d_r(M)=1$. However, the theorem itself holds equivalently for each of these two assumptions; see Remark \ref{rem:coprime}.} \cite[Theorem 1]{KSVGR}. This extension has found several applications. In particular it immediately implies a corresponding extension of Szemer\'edi's theorem to all finite abelian groups, since a matrix characterizing arithmetic progressions of a fixed length satisfies the above assumption; other applications include those in  \cite[Section 10]{ST} and \cite{SV}. Assuming that $d_r(M)=1$ is a simple way to ensure that the set of solutions has the `expected  dimension'; more precisely, we then have $\ker_G M \cong G^{m-r}$, as can be seen using the Smith normal form of $M$ (see \cite[Theorem II.9]{Newman}). We shall say more about this assumption in Section \ref{section:Remarks} below.
 
Some recent works have made use of removal results in the setting of infinite compact abelian groups. For instance, in \cite{CS2} it was shown that Theorem \ref{thm:KSVremlem} implies an analogous result for the circle group $G=\R/\Z$, formulated in terms of Haar measure, which was found to be useful for certain additive-combinatorial questions studied in $\Z_p$ as $p\rightarrow \infty$; see also \cite{CS1}. At the end of \cite{CS2}, the possibility of a removal result for a general compact abelian group was raised.
 
 The main result of this paper is an extension of Theorem \ref{thm:KSVremlem}, for matrices of determinantal 1, to all compact abelian groups. Below we discuss further motivation for this extension, but before that let us state the result formally.
 
All topological groups in this paper are assumed to be Hausdorff. Any compact group $G$ admits a unique Haar probability measure, which we denote by $\mu_G$. A subset of $G$ is said to be  \emph{Haar measurable} (or just \emph{measurable}) if it is in the completion of the Borel $\sigma$-algebra on $G$ relative to  $\mu_G$. Given a compact abelian group $G$ and a matrix $M\in \Z^{r\times m}$, the kernel $\ker_G M$ of the continuous homomorphism $M:G^m\to G^r$ is a compact subgroup of $G^m$, with its own Haar probability $\mu_{\ker_G M}$. For a measurable set $A\subset G$, the quantity $\mu_{\ker_G M}(A^m\cap \ker_G M)$ gives the natural notion of the  proportion (or density) of solutions contained in $A^m$.  This makes the setting of compact abelian groups a very natural one in which to seek general versions of results such as Theorem \ref{thm:StrongZpSzem} (note that if $G$ is finite then $\mu_{\ker_G M}(A^m\cap \ker_G M)$ is just $|A^m \cap \ker_G M|/|\ker_G M|$). For more background on the Haar measure, we refer the reader to \cite{D&E,H&R,Rud2}.
 
We can now state our main result.

\begin{theorem}\label{thm:cag-rem-lem}
Let $M \in \Z^{r \times m}$ satisfy $d_r(M)=1$. For any $\epsilon>0$, there exists $\delta=\delta(\epsilon,M)>0$ such that the following holds. If $A_1,A_2,\ldots,A_m$ are Borel subsets of a compact Hausdorff abelian group $G$ such that
$\mu_{\ker_G M}\big(A_1\times\cdots\times A_m\cap \ker_G M \big) \leq \delta$, then there exist Borel sets $R_1\subset A_1,\ldots, R_m\subset A_m$ such that $\mu_G(R_j)\leq \epsilon$ for all $j\in [m]$ and $\big(\prod_{j\in [m]} A_j\setminus R_j\big)\cap \ker_G M = \emptyset$.
\end{theorem}

We shall deduce this result from a more precise version, which holds for second countable compact abelian groups, and which gives additional information on the location of the sets $R_j$ and on their measure; see Theorem \ref{thm:cag-rem-lem+}. Note that Theorem \ref{thm:cag-rem-lem} also implies the inhomogeneous version of itself, where instead of $\ker_G M$ we consider the set of solutions $x\in G^m$ to $Mx=b$ for some non-zero $b\in G^r$.

From Theorem \ref{thm:cag-rem-lem}, one deduces directly the following generalization of Szemer\'edi's theorem (for a proof see the end of Section \ref{section:CayleyHypDefn}).

\begin{theorem}\label{thm:cag-Szem}
Let $M \in \Z^{r \times m}$ be invariant and satisfy $d_r(M)=1$. Then for any $\alpha>0$ there exists $c=c(\alpha,M)>0$ such that if $A$ is a measurable subset of a compact abelian group $G$ with $\mu_G(A)\geq \alpha$, then $\mu_{\ker_G M} (A^m\cap \ker_G M) \geq c$. 
\end{theorem}

In particular, for any positive integer $k$, any measurable set $A\subset G$ with $\mu_G(A)\geq \alpha>0$ satisfies\footnote{The case $k=3$ of this result, namely Roth's theorem for a general compact abelian group, can be treated using Fourier analysis; see for instance  \cite{Taoblog}.}
\[
\int_G \int_G \,1_A(x)\; 1_A(x+r)\;\cdots\; 1_A(x+(k-1)r)\;\ud\mu_G(x)\ud\mu_G(r) \geq c,
\]
where the positive lower bound $c=c(\alpha,k)$ is independent of the particular structure of $A$ and is in fact valid uniformly for all $G$.

In addition to the generality of Theorem \ref{thm:cag-rem-lem},  this extension to compact abelian groups offered us the motivation that it does not seem to follow from the known finite results by a simple measure-theoretic argument. Significant additive-combinatorial aspects had to be taken into account, requiring in particular further understanding of the relationship between combinatorial removal results for hypergraphs and their arithmetic counterparts. Let us complete this introduction by  detailing these points.

In order to prove a removal result in an infinite compact abelian group, it is natural to try to deduce it from a finitary version by a discretization argument. An approach of this type was taken in \cite{CS2}, yielding the above-mentioned analogue of Theorem \ref{thm:KSVremlem} for the circle group. However, as noted at the end of that paper, for more general compact abelian groups this approach yields a version of Theorem \ref{thm:cag-rem-lem} with a parameter $\delta$ depending on the topological dimension of the group. By contrast, the function $\delta$ in Theorem \ref{thm:cag-rem-lem} is independent of the compact abelian group. To obtain this, the approach in this paper consists instead in finding infinite analogues of some elements from known proofs of finite removal results, and combining those with some new elements in the infinite setting.

Most of the known proofs in the finite setting proceed by reducing the arithmetic removal result somehow to its combinatorial counterpart for uniform hypergraphs, a method which first appeared explicitly,  using graph removal lemmas, in \cite{KSV1}.

The most elaborate form of this method so far, i.e. the proof of \cite[Theorem 1]{KSVGR}, is implemented in a way that makes important use of properties specific to finite abelian groups, in particular the fact that multiplication by an integer does not increase the measure of a set in such a group (these aspects are discussed in more detail in Section 4 below). This prevents a simple transfer of the whole argument from \cite{KSVGR}  to the infinite setting, although several tools from that argument do transfer and are used in this paper.

The above-mentioned method is implemented in another way in the approach to arithmetic removal results given in \cite{SzegRem}. The main result of that paper is a so-called \emph{symmetry-preserving} version of the  removal lemma for finite hypergraphs. This version has the additional information that if the edge sets of the given hypergraph were invariant under a certain group action, then the edge sets to be removed can be guaranteed also to be invariant. This version of the hypergraph removal lemma turns out to have a useful  extension to the infinite setting, which we prove in this paper; see Lemma \ref{lem:simrem}. This extension concerns hypergraphs defined on general probability spaces and acted upon in a certain way by a compact group; see Definitions \ref{defn:t-action} and \ref{defn:InvHyp}. This infinite symmetry-preserving removal lemma gives a convenient footing for a proof of Theorem \ref{thm:cag-rem-lem+}. However, completing the  proof requires finding how to associate such an invariant hypergraph with a given system of linear equations on a compact abelian group. Indeed, in \cite{SzegRem} the finite symmetry-preserving removal lemma was shown to yield finite arithmetic removal results, but this was demonstrated only for certain examples of linear configurations, and it was not clear how to handle more general systems. In this paper, to clarify this we define a notion of a \emph{hypergraph representation} of a system of linear equations on an abelian group. This notion  extends and unifies previous finitary notions of a similar kind \cite{C,KSV2,Shap}, and it is designed to go together with the symmetry-preserving removal lemma; see Definition \ref{defn:Cayley-rep}. More precisely, this representation is a homomorphism which enables us to associate a certain measurable invariant hypergraph to the given system of equations, in such a way that the desired arithmetic removal result can be deduced from the removal lemma for this hypergraph; see Definition \ref{defn:Cay-hyp}.

In Section \ref{section:SymPresRem}, we prove the symmetry-preserving removal lemma. In Section \ref{section:CayleyHypDefn}, we define the hypergraph representation and use it to deduce the arithmetic removal result as mentioned above. In Section \ref{section:FindRep} we show that for any matrix $M\in \Z^{r\times m}$ with $d_r(M)=1$ and any compact abelian group, there exists such a hypergraph representation. 
In Section \ref{section:Remarks} we end with some remarks on potential further extensions of Theorem \ref{thm:cag-rem-lem}.

\section{A symmetry-preserving removal lemma for measurable hypergraphs}\label{section:SymPresRem}

In this section we establish the main result that we shall use concerning measurable hypergraphs, namely the symmetry-preserving removal lemma (Lemma \ref{lem:simrem}). This generalizes \cite[Theorem 2]{SzegRem}. Let us set up some terminology and notation.

Let $[ t ]=\{1,2,\ldots,t\}$, and let us denote the set of subsets of $[ t ]$ of size $k$ by $\binom{[t]}{k}$. Given any cartesian product $\prod_{i\in [t]} V_i$, and any set $e \subset [t]$, we denote by $p_e$ the projection $\prod_{i\in [t]}V_i \to \prod_{i\in e}V_i$ to the components indexed by $e$, thus $p_e(v) = ( v(i) )_{i\in e}$. (If $e$ is a singleton $\{i\}$ we write $p_i$ rather than $p_{\{i\}}$.) When there is no danger of confusion, we shall often use the notation $V_e$ to refer to the product $\prod_{i\in e} V_i$. 

The kind of hypergraph that we consider is the following. 

\begin{defn}
A \emph{$t$-partite $m$-colored $k$-uniform hypergraph}, or $(t,m,k)$-graph for short, is a triple $(V, C,E)$ consisting of the following elements. The \emph{vertex set} $V$ is the disjoint union of labelled sets $V_1,V_2,\ldots, V_t$. The set $C$ of \emph{edge color-classes} is a collection of $m$ distinct labelled sets $C_1,\ldots,C_m \in \binom{[t]}{k}$. The \emph{edge set} $E$ is the union of  sets $E_1,\ldots, E_m$ where each $E_j$ is a subset of $\prod_{i\in C_j} V_i$, the elements of which are the \emph{edges of color} $j$.
\end{defn}

We say that a $(t,m,k)$-graph is \emph{measurable} if there is a probability space structure $(V_i,\cV_i,\mu_i)$ on each vertex set $V_i$ (here $\cV_i$ denotes a $\sigma$-algebra of subsets of $V_i$, and $\mu_i$ a probability on $\cV_i$),  and  every set $E_j$ is in the product $\sigma$-algebra $\prod_{i\in C_j} \cV_i$. All the $(t,m,k)$-graphs that we consider in this paper are assumed to be measurable.

Given probability spaces $(V_i,\cV_i,\mu_i)$, $i\in [t]$, for any $e\subset [t]$ of size $|e| >1$ we shall denote by $(V_e,\cV_e,\mu_e)$ the product probability space $(\prod_{i\in e}V_i, \prod_{i\in e}\cV_i, \prod_{i\in e} \mu_i)$.

\begin{defn}[$(t,m,k)$-graph homomorphism]\label{defn:hom}
Let $H_1$ be a $(t,m,k)$-graph with vertex set $U=\bigsqcup_i U_i$, and let $H_2$ be a $(t,m,k)$-graph with vertex sets $V=\bigsqcup_i V_i$. A \emph{homomorphism} from $H_1$ to $H_2$ is a map $\phi:U\to V$ defined by $\phi(u)=\phi_i(u)$ for $u\in U_i$, where $(\phi_i)_{i\in [t]}$ is a $t$-tuple of measurable maps $\phi_i : U_i\to V_i$ with the following property: if $(u_i)_{i\in C_j} $ is an edge of $H_1$, then the image  $\big(\phi_i(u_i)\big)_{i\in C_j}$ is an edge of $H_2$.
\end{defn}
We say that $H_2$ is $H_1$-\emph{free} if there is no injective homomorphism $\phi:H_1\to H_2$. A measurable $(t,m,k)$-graph is \emph{finite} if the vertex sets $V_i$ are finite and the probabilities $\mu_i$ are uniform. In this paper we will only use homomorphisms from a finite $(t,m,k)$-graph to a possibly infinite $(t,m,k)$-graph. It is helpful to view these homomorphisms as points in the space $V_1^{U_1}\times V_2^{U_2}\times \cdots \times V_t^{U_t}$. Indeed, this leads naturally to the following definition of the homomorphism density, using the product probability on this space.

\begin{defn}
Let $F$ be a finite $(t,m,k)$-graph with vertex sets $U_i$, and let $H$ be a $(t,m,k)$-graph with vertex sets $V_i$. The  \emph{homomorphism density} of $F$ in $H$, denoted $\hd(F,H)$, is the probability that for a random $t$-tuple of maps $(\phi_i : U_i \to V_i)_{i\in [t]}$ the corresponding map $\phi$ is a homomorphism.\end{defn}

In particular, if $H$ has color-classes $C_1,\ldots, C_m$ and $F$ is the finite hypergraph with vertex set $[t]$ and edges $C_1,\ldots,C_m$, then, recalling that $(V_{[t]},\cV_{[t]},\mu_{[t]})$ denotes the product of the probability spaces $(V_i,\cV_i,\mu_i)$, we have
\begin{equation}\label{eq:homdens}
\hd(F,H)= \int_{V_{[t]}} \; \prod_{j\in [m]}\; 1_{E_j}\big( p_{C_j}(v)\big) \ud  \mu_{[t]} (v).
\end{equation}

For reasons that will become clear in the following sections, in this paper we only need this type of homomorphism $\phi: F\to H$ where each vertex class of $F$ is a singleton $U_i=\{i\}$. Note that any such homomorphism is an injective map, since the vertex classes of $H$ are disjoint by definition. We may sometimes refer to the image $\phi(F)=(\phi(i))_{i\in [t]}$ as a \emph{copy} of $F$ in $H$. In the general case, where $F$ may have more than one vertex per class, there is a similar but more complicated version of formula \eqref{eq:homdens}, but as mentioned above we shall not use this.

In the next subsection we shall obtain a removal lemma for $(t,m,k)$-graphs, Lemma \ref{lem:measrem}, by deducing it from the well-known removal lemma for finite hypergraphs. We shall then add the symmetry-preserving property in subsection \ref{secn:sympres}, obtaining the main result of this section, Lemma \ref{lem:simrem}.

\subsection{A removal lemma for $(t,m,k)$-graphs}
In this subsection we establish the following result.

\begin{lemma}\label{lem:measrem}
Let $t\geq k\geq 2$ and $m$ be positive integers, and let $0<\epsilon< 1$. There exists $\delta=\delta(t,k,\epsilon) >0$ such that the following holds. Let $H$ be a $(t,m,k)$-graph with vertex sets $V_i$, $i\in [t]$, and edge color-classes $C_j$, $j\in [m]$, let $F$ be the $(t,m,k)$-graph with vertex set $[t]$ and edges $C_j$, and suppose that $\hd(F,H)\leq \delta$. Then for each $j\in [m]$ there exists a measurable set $R_j\subset E_j(H)$ with $\mu_{C_j} (R_j) \leq \epsilon$, such that removing each $R_j$ from $E_j(H)$ yields an $F$-free $(t,m,k)$-graph.
\end{lemma}

The finite version of this result, that is the special case in which both $F$ and $H$ are finite $(t,m,k)$-graphs, is a version of the well-known hypergraph removal lemma, given for instance in \cite{Tao}. Our task here is to show that the above version for arbitrary probability spaces follows from the finite version. To prove this we use a discretization argument whereby $H$ is approximated by a $(t,m,k)$-graph $H^{(1)}$ whose vertex sets are partitioned into finitely many parts, and whose edge sets are disjoint unions of products of some of these parts. Then, we model each of these parts by a finite set of vertices, the cardinality of which is chosen according to the measure of the part. This enables us to relate $\hd(F,H)$ with $\hd(F,H^{(2)})$ for some associated finite $(t,m,k)$-graph $H^{(2)}$, thus reducing the proof to an application of the finite version of Lemma \ref{lem:measrem}.

\begin{proof}[Proof of Lemma \ref{lem:measrem}]
Let $\delta'\leq \epsilon/(4m)$ be such that the finite version of Lemma \ref{lem:measrem} holds with parameters $\epsilon/(4m),t,k$. (As mentioned above, this finite version is known; indeed it is essentially \cite[Corollary 1.14]{Tao}.)
Suppose that $\hd(F,H)\leq \delta$ with $\delta=\delta'/2$.\\
\indent For each $j\in [m]$, since the $\sigma$-algebra $\cV_{C_j}$ on $V_{C_j}=\prod_{i\in C_j} V_i$ is generated by products of measurable subsets of the components $V_i$, there exist disjoint sets $B_{j,1},B_{j,2},\ldots, B_{j,M_j}$, each of the form $B_{j,r}=\prod_{i\in C_j} D_{i,j,r}$ with $D_{i,j,r}\in \cV_i$, satisfying 
\begin{equation}\label{eq:approx1}
\mu_{C_j}\left(E_j(H) \; \Delta \; \bigsqcup_{r=1}^{M_j} B_{j,r}\right) \leq \delta/m \leq \epsilon/2.
\end{equation}
Let $H^{(1)}$ be the $(t,m,k)$-graph obtained from $H$ by replacing the edge sets $E_j(H)$ with $E^{(1)}_j:=  \bigsqcup_{r=1}^{M_j} B_{j,r}$. By \eqref{eq:approx1} and a simple telescoping argument using multilinearity of the function $(1_{E_1},\ldots,1_{E_m})\mapsto \prod_{j\in [m]}1_{E_j}\circ p_{C_j}$, we have\footnote{We illustrate the argument for $m=3$: for any functions $f_j,g_j:V_{[t]}\to \R$, $j\in [3]$, we have $f_1 f_2 f_3 = (f_1-g_1) f_2 f_3 + g_1 f_2 f_3 = (f_1-g_1) f_2 f_3 + g_1 (f_2-g_2) f_3+ g_1  g_2 f_3 = (f_1-g_1) f_2 f_3 + g_1 (f_2-g_2) f_3 + g_1 g_2 (f_3- g_3)
+ g_1 g_2 g_3$; we then apply this with $f_j=1_{E^{(1)}_j}\circ p_{C_j}$ and $g_j=1_{E_j}\circ p_{C_j}$.}
\[
\hd(F,H^{(1)})\leq \hd(F,H)+m \frac{\delta}{m}\leq \delta'.
\]
We shall now show that $H^{(1)}$ can be made $F$-free by removing a set of measure at most $\epsilon/2$ from each set $E_j^{(1)}$.\\
\indent For each $i\in [t]$ we define a partition of $V_i$ generated by all the sets $D_{i,j,r}$.  More precisely,  let $\cP_i$ denote the partition of $V_i$ into the atoms of the finite $\sigma$-algebra generated by the collection of sets $\bigcup_{j\in [m]: \,C_j\ni i} \{D_{i,j, r}: r \in [M_j]\}$. 
Let $K_i= |\cP_i|$, thus $\cP_i=\{ P_{i,1}, P_{i,2},\ldots, P_{i,K_i}\}$. Each set $E_j^{(1)}$ is a disjoint union of sets of the form $\prod_{i \in C_j} P_{i,\ell_i}$ for some $\ell=(\ell_i)_{i\in C_j} \in \prod_{i\in C_j} [K_i]$. Thus $H^{(1)}$ can already be viewed as a finite hypergraph, with vertex sets  $\cP_1,\ldots, \cP_t$ and edges these $k$-tuples $\ell$. However, the measures of the atoms $P_{i,j}$ are not necessarily equal, so the probabilities on the vertex sets of this hypergraph may fail to be uniform. In order to apply the finite version of the removal lemma, we shall now approximate this weighted hypergraph by a finite $(t,m,k)$-graph $H^{(2)}$.\\
\indent Note that if $v=\big(v(1),\ldots,v(t)\big)$ is a copy of $F$ in $H^{(1)}$,  with $v(i)\in P_{i,r_i}\subset V_i$ for each $i\in [t]$, then in fact every point in $P_{1,r_1}\times \cdots \times P_{t,r_t}$  is such a copy, and this product set gives us a measure $\mu_1(P_{1,r_1})\cdots \mu_t(P_{t,r_t})$ of homomorphisms $F\to H^{(1)}$.\\
\indent Let $N$ be a large positive integer to be determined below,  depending on $k,t,\epsilon$ and the measure of the atoms $P_{i,r_i}$.\\
\indent Let $H^{(2)}$ be the finite $(t,m,k)$-graph defined as follows.   The finite vertex sets, denoted $V_1',\ldots,V_t'$, are each of  cardinality $N$, with uniform probability denoted $\mu_i'$. Each set $V_i'$ is partitioned into sets $Q_{i,0},Q_{i,1},\ldots, Q_{i,K_i}$, such that we have
\begin{equation}
\forall\, r\in [K_i],\; |Q_{i,r}|= q_{i,r},\;\; q_{i,r}/N \leq \mu_i(P_{i,r})< (q_{i,r}+1)/N,\;\textrm{ and }\; |Q_{i,0}|= q_{i,0}\leq K_i.
\end{equation}
The edge color classes of $H^{(2)}$ are the same as for $H^{(1)}$, and for each such class $C_j$ the edge set $E_j^{(2)}$ of $H^{(2)}$ is defined as follows. A $k$-tuple $(v(i))_{i\in C_j}\in V_{C_j}'$ is an edge in $E_j^{(2)}$ if and only if 
 $\prod_{i\in C_j} P_{i,r_i}\subset E_j^{(1)}$, where $v(i)\in Q_{i,r_i}$ for each $i\in C_j$. In other words, $E_j^{(2)}$ is the disjoint union of all the sets $\prod_{i\in C_j} Q_{i,r_i}$ satisfying $\prod_{i\in C_j} P_{i,r_i} \subset E_j^{(1)}$.\\
\indent Since each set $Q_{i,r}$ satisfies $\mu_i'(Q_{i,r})\leq \mu_i(P_{i,r})$, we have $\hd(F,H^{(2)})\leq \hd(F,H^{(1)})\leq \delta'$. By the finite version of the removal lemma, there exist sets $R_j'\subset E_j^{(2)}$ with $\mu_{C_j}'(R_j')\leq \epsilon/(4m)$ such that, removing each $R_j'$ from $E_j^{(2)}$, the resulting hypergraph is $F$-free.\\
\indent Let us now use the sets $R_j'$ to specify which subsets to remove from $E_j^{(1)}$. To do so, we first show that each $R_j'$ may be replaced with a set $R_j''$ that is a union of sets of the form $\prod_{i\in C_j} Q_{i,r_i}$, in such a way that the sets $R_j''$ still have small measure and preserve the removal property.\\
\indent Let $R_j''$ be the union of sets $\prod_{i\in C_j} Q_{i,r_i}$ such that 
\[
\Big|R_j'\cap \prod_{i\in C_j} Q_{i,r_i}\Big| \geq m^{-1} \prod_{i\in C_j} q_{i,r_i}.
\] 
We have $|R_j''|\leq m |R_j'|$, and so $\mu'_{C_j}(R_j'')\leq m\, \mu'_{C_j}(R_j')\leq \epsilon/4$.\\
\indent We claim that removing $R_j''$ (instead of $R_j'$) from $E_j^{(2)}$ still yields an $F$-free $(t,m,k)$-graph. Indeed,  suppose that $v_0\in \prod_{i\in [t]}Q_{i,r_i}$ is a copy of $F$ in $H^{(2)}$. Then, by the definition of $H^{(2)}$, every element $v\in \prod_{i\in [t]}Q_{i,r_i}$ is such a copy. By the removal property of the sets $R_j'$, for any such $v$ there   exists $j\in [m]$ such that the edge $p_{C_j}(v)$ lies in $R_j'$. There must therefore exist $j\in [m]$ such that there are at least $m^{-1}\prod_{i\in [t]}q_{i,r_i}$ such copies $v$ with $p_{C_j}(v)\in R_j'$. On the other hand, an edge $w\in \prod_{i\in C_j}Q_{i,r_i}$ can satisfy $w=p_{C_j}(v)$ for at most $\prod_{i\in [t]\setminus C_j}q_{i,r_i}$ of these copies $v$. We therefore conclude that $\Big|R_j'\cap \prod_{i\in C_j} Q_{i,r_i}\Big| \geq m^{-1} \prod_{i\in C_j} q_{i,r_i}$. Hence all these copies (including $v_0$) have $p_{C_j}(v)\in R_j''$ and are therefore eliminated by removing $R_j''$. This proves our claim.\\
\indent We can now specify the sets $R_j^{(1)}$ that we remove from $E_j^{(1)}$. Let $R_j^{(1)}$ be the union of sets $\prod_{i\in C_j} P_{i,r_i} \subset E_j^{(1)}$ such that  $\prod_{i\in C_j}Q_{i,r_i} \subset R_j''$. Note that 
\begin{equation}\label{eq:measbounds}
\mu_{C_j}\Big(\prod_{i\in C_j}P_{i,r_i}\Big)
\leq \prod_{i\in C_j} \Big(\mu'_i (Q_{i,r_i})+\frac{1}{N}\Big)\leq \mu'_{C_j} \Big(\prod_{i\in C_j}Q_{i,r_i}\Big)+\frac{2^k}{N}.
\end{equation}
Choosing $N>\frac{2\cdot 2^k}{\epsilon} \max_{j\in[m]} \left(\prod_{i\in C_j} K_i\right)$, we deduce from \eqref{eq:measbounds} that
\begin{equation}\label{eq:rem_mes}
\mu_{C_j}\Big(R_j^{(1)}\Big)\leq \mu_{C_j}'\Big(R_j''\Big)+\Big(\prod_{i\in C_j}K_i\Big) \frac{2^k}{N}\leq \frac{\epsilon}{2}, \textrm{ for each }j\in [m].
\end{equation}
If there was a copy $v$ left in $\bigcap_{j\in [m]}p_{C_j}^{-1}\Big(E^{(1)}\setminus R_j^{(1)}\Big)$, then there would have to be in fact a measure $\mu_{[t]}\Big(\prod_{i\in [t]}P_{i,r_i}\Big)$ of such copies, where $v(i)\in P_{i,r_i}$ for each $i\in [t]$. Therefore, by an analogue of \eqref{eq:measbounds}, there would be a measure at least $\mu_{[t]}\Big(\prod_{i\in [t]}P_{i,r_i}\Big)-\frac{2^t}{N}$ of copies of $F$ in $ \bigcap_{j\in [m]}p_{C_j}^{-1}\Big(E_j^{(2)}\setminus R_j^{(2)}\Big)$. If
\[
N > 2^t/\min\left\{ \mu_{[t]}\big(\prod_{i\in [t]}P_{i,r_i}\big): (r_i)\in \prod_{i\in C_j}[K_i],\; j\in [m]\right\},
\]
then there is at least one such copy of $F$, contradicting the removal property of the sets $R_j^{(2)}$.\\
\indent We now set $R_j=R_j^{(1)}\cup \Big(E_j(H) \setminus  \bigsqcup_{\ell=1}^{M_j} B_{j,\ell}\Big)$, which by \eqref{eq:approx1} and \eqref{eq:rem_mes} has measure at most $\epsilon$ for each $j\in[m]$, and the proof is complete.
\end{proof}

\begin{remark}\label{rem:moregenrem}
Lemma \ref{lem:measrem} concerns the so-called `partite hypergraph version' of the removal lemma (as it is called in \cite{Tao}), which corresponds to the case of formula \eqref{eq:homdens} in which $F$ has one vertex per class. This case suffices for our purposes in this paper, as we shall see in the next sections. Let us mention that there is a version of Lemma \ref{lem:measrem} where $F$ may have more than one vertex in each part $U_i$, and that in fact this extension can be deduced using Lemma \ref{lem:measrem}.
\end{remark}

\subsection{Preserving symmetries}\label{secn:sympres}

We now move on to the main result of this section, Lemma \ref{lem:simrem}. This is a version of Lemma \ref{lem:measrem} which preserves certain symmetries of the given hypergraph. The symmetries of $(t,m,k)$-graphs that we shall consider are described in terms of a type of group action on the product of the vertex sets, that we call a $t$-\emph{partite action} (see Definition \ref{defn:t-action}). To build up to this notion, we first recall the definition of a measurable group action (see for instance \cite[\S 3]{Var}). We denote the identity element of a group $G$ by $\id_G$.

\begin{defn}[Group action on a probability space]\label{defn:action}
Let $(V,\cV ,\mu)$ be a probability space, and let $\cG$ be a group. An \emph{action} of $\cG$ on $V$ is a map $\Phi: \cG\times V\to V$ satisfying the following properties:
\begin{enumerate}
\item $\forall\, v\in V$, $\forall\,g,h\in \cG$ we have $\Phi(gh,v) = \Phi(g,\Phi(h,v))$, and $\Phi(\id_{\cG},v)=v$.
\item For each $g\in \cG$ the invertible map $\Phi_g: v\mapsto \Phi(g,v)$ is measurable and preserves $\mu$, that is for any set $A\in \cV$, we have $\Phi_g^{-1}(A)\in \cV$ and $\mu(\Phi_g^{-1}(A) )= \mu(A)$. 
\end{enumerate}
In other words, the map $g\mapsto \Phi_g$ is a homomorphism from $\cG$ into the group of measure-preserving automorphisms of $V$.  If $\cG$ is a topological group, with Borel $\sigma$-algebra denoted $\cB_{\cG}$, then we say that the action $\Phi$ is  \emph{measurable} if the map $\Phi$ is measurable from $(\cG\times V,\cB_{\cG}\times \cV)$ to $(V,\cV)$.
\end{defn}

We shall often use the simpler notation $g\cdot v$ for $\Phi(g,v)$.

Given an action of $\cG$ on $(V,\cV,\mu)$, a set $B\in \cV$ is said to be $\cG$-\emph{invariant} if $g\cdot B = B$ for all $g\in \cG$. These sets form a sub-$\sigma$-algebra of $\cV$ that we denote by $\cE_{\cG}$. A measurable function $f:V\to \R$ is said to be $\cG$-\emph{invariant} if, for every $g\in \cG$, we have $f(g\cdot v)=f(v)$ for all $v\in V$. This is equivalent to $f$ being measurable with respect to $\cE_{\cG}$.\\
\indent In this paper we consider measurable actions mainly of \emph{compact} groups. We shall use the following simple notion of the average of a measurable function with respect to such an action. (We shall only need to take the average of non-negative functions.)

\begin{defn}
Let  $(V,\cV ,\mu)$ be a probability space, let $\cG$ be a compact group with Haar probability measure $\mu_{\cG}$, and let $\Phi:\cG\times V\to V$ be a measurable action. Then, for any non-negative measurable function $f:V\to \R$, we denote by $\cT_{\cG}(f)$ the non-negative measurable function defined by $\cT_{\cG}(f)(v)= \int_{\cG} f(g^{-1}\cdot v) \ud\mu_{\cG}(g)$.
\end{defn}

From our assumptions we have that the function $(g,v)\mapsto f(g^{-1}\cdot v)$ is $(\cB_{\cG}\times \cV)$-measurable. By Fubini's theorem \cite[Theorem 8.8]{Rud}, we therefore have that $\cT_{\cG}(f)$ is indeed a $\cV$-measurable function, and satisfies
\begin{equation}\label{eq:Fubini}
\int_V  \cT_{\cG}(f)(v) \ud\mu(v) = \int_{\cG} \left(\int_V f(g^{-1}\cdot v) \ud\mu(v)\right)\ud\mu_{\cG}(g) = \int_{\cG\times V} f(g^{-1}\cdot v)\ud (\mu_{\cG} \times\mu).
\end{equation}
Note also that for any non-negative measurable functions $f,g$ on $V$ we have $\cT_{\cG}(f+g)=\cT_{\cG}(f)+\cT_{\cG}(g)$, and in particular if  $f\geq g$ then $\cT_{\cG}(f)\geq \cT_{\cG}(g)$.\\
\indent A more general notion of averaging can be given in terms of the conditional expectation relative to the $\sigma$-algebra $\cE_{\cG}$, but the above definition is more convenient for us. (We discuss this in Remark \ref{rem:condexp}.)

\begin{defn}[$t$-partite action]\label{defn:t-action}
Let $(V_i,\cV_i,\mu_i),i\in [t]$, be probability spaces, and let $\cG$ be a topological group. We say that an action $\Phi: \cG\times V_{[t]}\to V_{[t]}$ is a $t$-\emph{partite action} if it is of the following form: for each $i\in [t]$ there is a topological group $G_i$ with a measurable action $\Phi_i: G_i\times V_i\to V_i$, such that $\cG$ is a closed subgroup of $G_1\times\cdots\times G_t$ (in the product topology) and for every $g\in \cG,v\in V_{[t]}$ we have $\Phi(g,v) (i) = \Phi_i(g(i),v(i))$ for each $i\in [t]$.
\end{defn}

In the next section we shall focus on $t$-partite actions where each $V_i$ is a second-countable compact abelian group $G_i$ acting on itself by addition. For the main results of this section, however, we can work with more general $t$-partite actions of compact groups. Let us record the following basic fact.
\begin{lemma}\label{lem:t-act-meas}
A $t$-partite action is a measurable action.
\end{lemma}
\begin{proof}
The fact that a $t$-partite action $\Phi$ is indeed an action is  straightforward. To see that the measurability of each map $\Phi_i$ implies measurability of $\Phi$, it suffices to check this for an arbitrary product set $A=A_1\times\cdots\times A_t$, $A_i\in \cV_i$. To this end we note that  $\Phi^{-1} A = (\cG\times V_{[t]})\;\cap\;\cR\Big(\prod_i \Phi_i^{-1} A_i\Big)$, where $\cR:\prod_i (G_i\times V_i)\to \Big(\prod_i G_i\Big) \times V_{[t]}$ is the map permuting the coordinates appropriately. We can then use the fact that each $\Phi_i^{-1} A_i$ lies in $\cB_{G_i}\times \cV_i$ to deduce that $\Phi^{-1} A$ lies in $\cB_{\cG}\times \cV_{[t]}$.
\end{proof}
Given a $t$-partite action of a compact group $\cG$ on $V_{[t]}$, and a non-empty set $e\subset [t]$, we denote by $\cG_e$ the closed subgroup $p_e(\cG)$ of $\prod_{i\in e} G_i$. Recall that the map $p_e$ is the coordinate projection corresponding to $e$. On the direct product $G_1\times\cdots\times G_t$, this map is a continuous homomorphism onto $\cG_e$. We can then define a measurable action $\Phi_e: \cG_e\times V_e\to V_e$ by $\Phi_e(g,v) (i) = \Phi_i(g(i),v(i))$. 

\begin{defn}[Invariant $(t,m,k)$-graph]\label{defn:InvHyp}
Let $(V_i,\cV_i,\mu_i),i\in [t]$, be probability spaces, and let $\cG$ be a topological group with a $t$-partite action $\cG\times V_{[t]}\to V_{[t]}$. A $(t,m,k)$-graph $H$ with vertex sets $V_i$ is said to be $\cG$\emph{-invariant} if for each $j\in [m]$, the edge set $E_j(H)$ is $\cG_{C_j}$-invariant.
\end{defn}

We shall use the following fact that relates averaging over $\cG$ to averaging over $\cG_e$, for each projection $p_e$. 

\begin{lemma}\label{lem:proj}
Let $(V_i,\cV_i,\mu_i),i\in [t]$, be probability spaces, and let $\cG\times V_{[t]}\to V_{[t]}$ be a $t$-partite action by a compact group $\cG$ with Haar probability. Then for any $e\in \binom{[t]}{k}$, for any non-negative measurable function $f: V_e\to \R$, we have
\begin{equation}\label{eq:proj}
\cT_{\cG} (f \circ p_e) = (\cT_{\cG_e} (f)) \circ p_e.
\end{equation}
\end{lemma}
\begin{proof}
The actions $\Phi,\Phi_e$ commute with $p_e$, that is we have
\begin{equation}\label{eq:t-partite}
p_e(g\cdot v)= p_e(g) \cdot p_e(v)\textrm{ for every }g\in \cG,v\in V_{[t]}.
\end{equation}
Moreover, the map $p_e:\cG\to \cG_e$ is a surjective continuous homomorphism. We therefore have $\mu_{\cG_e}=\mu_{\cG}\circ p_e^{-1}$, where  $\mu_{\cG},\mu_{\cG_e}$ are the Haar probabilities on $\cG,\cG_e$. Thus for any $v\in V_{[t]}$ we have
\begin{align*}
\cT_{\cG} (f \circ p_e) (v) & \;\;=  \int_{\cG} f\big(p_e(g^{-1}\cdot v)\big) \ud\mu_{\cG}(g)  \;\;\;\; = \;\; \int_{\cG} f\big(p_e(g)^{-1}\cdot p_e(v)\big) \ud\mu_{\cG}(g)\\
 & \;\;=  \int_{\cG_e} f\big(g_e^{-1}\cdot p_e(v)\big) \ud\mu_{\cG_e}(g_e) \;  =  \;\;\cT_{\cG_e}(f)\big(p_e(v)\big).  \qedhere
\end{align*}
\end{proof}

We can finally establish the main result of this section.

\begin{lemma}[Symmetry-preserving removal lemma]\label{lem:simrem}\hfill\\
Let $t\geq k \geq 2$ and $m$ be positive integers, and let $\epsilon>0$. There exists $\delta=\delta(t,k,\epsilon)>0$ such that the following holds. Let $(V_i,\cV_i,\mu_i),i\in [t]$, be probability spaces, let $H$ be a $(t,m,k)$-graph with vertex sets $V_i$ and edge color-classes $C_j$, let $\cG\times V_{[t]}\to V_{[t]}$ be a $t$-partite action by a compact group $\cG$ such that $H$ is $\cG$-invariant, let $F$ be the $(t,m,k)$-graph on $[t]$ with edges $C_1,\ldots,C_m$, and suppose that $\hd(F,H)\leq \delta$. Then for each $j\in [m]$ there exists a measurable set $S_j\subset E_j(H)$ with $\mu_{C_j}(S_j)\leq\epsilon$, such that  removing $S_j$ from $E_j(H)$ for each $j\in [m]$ yields an $F$-free $(t,m,k)$-graph that is still $\cG$-invariant.
\end{lemma}
An equivalent version of the conclusion is that for each $j\in [m]$ there exists a $\cG_{C_j}$-invariant set $S_j\subset E_j(H)$ with $\mu_{C_j}(S_j)\leq\epsilon$, such that removing $S_j$ from $E_j(H)$ for each $j\in [m]$ yields an $F$-free $(t,m,k)$-graph.
\begin{proof}
Let $R_j\subset E_j(H),j\in [m]$, be the removal sets given by Lemma \ref{lem:measrem} applied with parameter $\delta$ such that $\mu_{C_j}(R_j)\leq \epsilon /(2 |E(F)|)=\epsilon /(2m)$. 

We define a new removal set $S_j\subset E_j(H)$ as follows:
\begin{equation}
S_j:= \{ v \in V_{C_j}: h_j (v) > 1/(2m)\},\textrm{ where } h_j:=\cT_{\cG_{C_j}}(1_{R_j}). 
\end{equation}
Note that $h_j$ is a $\cG_{C_j}$-invariant function, whence $S_j$ is a $\cG_{C_j}$-invariant measurable set. Moreover, we have $\mu_{C_j}(S_j)\leq \epsilon$. Indeed, by Markov's inequality and \eqref{eq:Fubini} we have
\[
\frac{\mu_{C_j}(S_j)}{2m} \leq  \int_{V_{C_i}} h_j (v)  \ud\mu_{C_j}(v)
 =  \int_{\cG_{C_j}} \left(\int_{V_{C_j}}1_{R_j} (g^{-1}\cdot v)  \ud\mu_{C_j}(v)\right) \ud\mu_{\cG_{C_j}}(g) =\mu_{C_j}(R_j).
\]
We now show that removing $S_j$ from $E_j$ for each $j\in [m]$ yields an $F$-free $(t,m,k)$-graph, i.e. that we have
\begin{equation}\label{eq:simremgoal}
\bigcap_{j\in [m]} p_{C_j}^{-1}(E_j\setminus S_j)=\emptyset.
\end{equation}
In other words, we show that the function $ \prod_{j\in [m]} 1_{E_j} \circ p_{C_j}$ is 0 everywhere on the region $\bigcap_{j\in [m]} p_{C_j}^{-1}(V_{C_j}\setminus S_j)$.

By a telescoping argument using multilinearity (similar to the one used in the proof of Lemma \ref{lem:measrem}), we have
\[
\prod_{j\in [m]} 1_{E_j} \circ p_{C_j} = \sum_{j\in [m]} 1_{R_j}\circ p_{C_j} \prod_{\ell\neq j} f_\ell \;\;+\;\;\prod_{j\in  [m]} (1_{E_j} - 1_{R_j})\circ p_{C_j},
\]
where for each $j$ we have $f_\ell= (1_{E_\ell}- 1_{R_\ell})\circ p_{C_\ell}$ if $\ell<j$, and  $f_\ell=1_{E_\ell}\circ p_{C_\ell}$ if $\ell>j$. It follows that 
\begin{eqnarray*}
 \prod_{j\in [m]} 1_{E_j} \circ p_{C_j} \leq  \sum_{j\in [m]} 1_{R_j}\circ p_{C_j} \;+ \;\prod_{j \in [m]} (1_{E_j} - 1_{R_j})\circ p_{C_j}\qquad \textrm{everywhere on } V_{[t]}.
\end{eqnarray*}
Let us now apply $\cT_{\cG}$ to both sides. Since each set $E_j$ is $\cG_{C_j}$-invariant, by \eqref{eq:t-partite} each set   $p_{C_j}^{-1}E_j$ is in the $\sigma$-algebra of $\cG$-invariant sets, and therefore so is their intersection. By the removal property of the sets $R_j$, we also have $\prod_{j \in [m]} (1_{E_j} - 1_{R_j})\circ p_{C_j}=0$ everywhere. Finally, by linearity and \eqref{eq:proj} we have $\cT_{\cG}\left(\sum_{j\in [m]} 1_{R_j}\circ p_{C_j} \right)=\sum_{j\in [m]}h_j \circ p_{C_j}$. Combining these facts, we conclude that 
\begin{equation}\label{eq:keyineq}
\prod_{j\in [m]} 1_{E_j} \circ p_{C_j} -  \sum_{j\in [m]}   h_j\circ p_{C_j} \leq 0 \quad\textrm{everywhere on } V_{[t]}.
\end{equation}
Now, on the region $\bigcap_{j\in [m]} p_{C_j}^{-1}(V_{C_j}\setminus S_j)$, the function $ \sum_{j\in [m]} h_j\circ p_{C_j}$ takes values at most $1/2$, by definition of the sets $S_j$. Therefore, if $\prod_{j\in [m]} 1_{E_j} \circ p_{C_j}(v)$ were positive for some $v$ in this region, then it would have to take value 1 at $v$ and then the left side of \eqref{eq:keyineq} would be positive at $v$, a contradiction.
\end{proof}

\begin{remark}\label{rem:condexp}
Recall that the conditional expectation relative to a sub-$\sigma$-algebra $\cE$ of $\cV$ can be defined on the Hilbert space $L^2(V,\cV,\mu)$ as the orthogonal projection to the closed subspace $L^2(V,\cE,\mu|_{\cE})$; the conditional expectation of $f\in L^2(V,\cV,\mu)$ relative to $\cE$ is denoted $\E(f|\cE)$. If a compact group $\cG$ with Haar probability has a measurable action on $(V,\cV,\mu)$ then one can show that $\cT_{\cG}$ agrees with the conditional expectation relative to the $\sigma$-algebra $\cE_{\cG}$ of $\cG$-invariant sets. More precisely, letting $f$ be any function class in $L^2(V,\cV,\mu)$, and letting $f'$ be any function in this class, we have that $\cT_{\cG}(f')$ is in the class $\E(f|\cE_{\cG})$ (this can be proved by showing that $\cT$ yields an orthogonal projection $L^2(V,\cV,\mu)\to L^2(V,\cE_{\cG},\mu|_{\cE_{\cG}})$). This conditional expectation relative to $\cE_G$ is defined even for actions that are not necessarily measurable. Thus one can obtain analogues of the results in this subsection for possibly non-measurable actions.  However, $\E(f|\cE_{\cG})$ defines a function only up to a null-set, and this introduces several additional technicalities. Arguments using $\cT_{\cG}$, as above,  are therefore more convenient for our purposes, in addition to being more explicit.\\
\indent In a similar vein, one can obtain analogues of the results in this section when each set $E_j(H)$ is only assumed to lie in the completion $\cV_{C_j}^*$ of $\cV_{C_j}$ relative to $\mu_{C_j}$. One can also define a group action $\Phi:\cG\times V\to V$ to be measurable in the weaker sense that $\Phi^{-1}$ takes values in the \emph{completion} of $\cB_{\cG}\times \cV$ relative to $\mu_{\cG}\times \mu$. One can then use the version of Fubini's theorem for completed product measures \cite[Theorem 8.12]{Rud}, but again this is less convenient for us.
\end{remark}

\section{Cayley $(t,m,k)$-graphs and systems of linear equations}\label{section:CayleyHypDefn}

Our aim now is to apply the results from the previous section to  prove Theorem \ref{thm:cag-rem-lem}. We shall in fact prove the following version first.

\begin{theorem}\label{thm:cag-rem-lem+}
Let $M \in \Z^{r \times m}$ satisfy $d_r(M)=1$. For any $\epsilon>0$, there exists $\delta=\delta(\epsilon,M)$, $0<\delta <1$, such that the following holds. Let $A_1,\ldots,A_m$ be Borel subsets of a second-countable compact Hausdorff abelian group $G$ such that
$\mu_{\ker_G M}\big(A_1\times\cdots\times A_m\cap \ker_G M \big) \leq \delta$, and for each $j\in [m]$ let $G^{(j)}$ denote the closed subgroup $p_j(\ker_G M)$ of $G$. Then for each $j\in [m]$ there exists a Borel  set $R_j\subset A_j\cap G^{(j)}$, such that $\mu_{G^{(j)}}(R_j)\leq \epsilon$ for all $j\in [m]$, and $\big(\prod_{j\in [m]} A_j\setminus R_j\big)\cap \ker_G M = \emptyset$.
\end{theorem}
The added information here is firstly that we only need to remove elements from $A_j$ that are in the projection $p_j(\ker_G M)$ of the solution space (this is quite clear intuitively and is also the case in Theorem \ref{thm:cag-rem-lem}). Secondly, each set $R_j$ is small not just in the measure $\mu_G$ but in the possibly larger measure $\mu_{G^{(j)}}$. Indeed, note that if the index $\kappa_j=|G:G^{(j)}|$ is finite then we must have $\mu_{G^{(j)}}= \kappa_j \cdot \mu_G|_{G^{(j)}}$, where $\mu_G|_{G^{(j)}}$ denotes the restriction of $\mu_G$ to $G^{(j)}$. Thus, while the conclusion $\mu_{G^{(j)}}(R_j)\leq \epsilon$ above is roughly equivalent to the conclusion $\mu_G(R_j)\leq \epsilon$ in Theorem \ref{thm:cag-rem-lem} if $\kappa_j=O(\epsilon^{-1})$, the former conclusion is stronger otherwise. We explain the use of second countability in Remark \ref{rem:2ndcount}.\\ 
\indent To prove Theorem \ref{thm:cag-rem-lem+}, we want to find, given a system of linear equations of determinantal 1 on $G$, a certain invariant hypergraph that represents the system in such a way that the theorem follows from Lemma \ref{lem:simrem}. In \cite{SzegRem}, a notion of a finite \emph{Cayley hypergraph} was introduced and shown to give a representation of the desired kind for certain systems of equations on certain finite abelian groups, but it was not clear how far this method could be extended. The main objective for the remainder of this paper is to show that there is a general version of this framework that can handle all systems of determinantal 1.

We begin with a definition analogous to \cite[Definition 2.1]{SzegRem}.

\begin{defn}[Cayley $(t,m,k)$-graph]\label{defn:Cay-hyp}
We call a $(t,m,k)$-graph $H$ a \emph{Cayley $(t,m,k)$-graph} if it has the following properties. For each $i\in [t]$, the $i$-th vertex set is a compact group $G_i$ with Borel $\sigma$-algebra and Haar probability,
 and there is a closed subgroup $\cG$ of the direct product $\prod_{i\in [t]} G_i$ such that $H$ is invariant under the $t$-partite action of $\cG$ on $G_{[t]}$, where each $G_i$ acts on itself by left-multiplication. To specify these properties, we write $H\in H_{k,t}((G_i),C,\cG)$, where $C$ is the set of edge color-classes of $H$.
\end{defn}
To illustrate this, let us note briefly how as a special case one finds ``bipartite Cayley graphs" (as they are called in \cite{GB}, for instance). Let $t=k=2$, $m=1$, $C=\{ \{1,2\}\}$, let $G_1=G_2=G$ be a finite group, and let $\cG=\{(g,g):g\in G\}$ be the diagonal subgroup of $G\times G$. Then $H\in H_{2,2}(G,C,\cG)$ means that the edge set $E(H)$ is a union of right cosets of $\cG$. Using the map $(g_1,g_2)\mapsto g_2^{-1}g_1$, we can identify the quotient $\cG\backslash E$ with a set $A\subset G$, and thus see that $H$ is the bipartite Cayley graph on $G_1\sqcup G_2$ generated by $A$ (that is we have $(g_1,g_2)\in E$ if and only if $g_2^{-1}g_1\in A$).
\begin{remark}\label{rem:2ndcount}
By Lemma \ref{lem:t-act-meas}, the $t$-partite action in Definition \ref{defn:Cay-hyp} is measurable (in the sense of Definition \ref{defn:action}) if the action of each $G_i$ on itself by left-multiplication is measurable. The latter measurability of the group operation holds for any second-countable group $G$ (that is a topological group such that the underlying topological space has a countable base). Indeed, by continuity of multiplication the preimage of a Borel set $A\subset G$ is Borel in $G\times G$, i.e. it lies in the Borel $\sigma$-algebra $\cB_{G\times G}$. By second countability, we have that $\cB_{G\times G}$ equals the product $\sigma$-algebra $\cB_G\times \cB_G$ (see \cite[Lemma 6.4.2]{Boga}), so the action is  measurable. Without second countability, the $\sigma$-algebra $\cB_{G\times G}$  may be strictly larger than $\cB_G\times \cB_G$ (see \cite[Example 6.4.3]{Boga}). These facts, together with other aspects (such as Lemma \ref{lem:Haarquotient} below), make second countability a  useful assumption in Theorem \ref{thm:cag-rem-lem+}. Moreover, once this theorem has been proved, Theorem \ref{thm:cag-rem-lem} can be deduced using an inverse limit argument. This is done in Appendix \ref{app:A}. Thus, from now on we shall consider such invariant hypergraphs only on \emph{second-countable} compact groups.
\end{remark} 

Defining Cayley $(t,m,k)$-graphs in terms of invariance, as above, relates them clearly to the previous section. To relate them to arithmetic removal results, it is useful to describe the edge sets of such hypergraphs in terms of generating sets.

\begin{lemma}\label{lem:CayHypEdges}
Let $G_1,G_2,\ldots, G_t$ be second-countable compact  groups, let $\cG$ be a closed subgroup of $G_{[t]}$, and let $H\in H_{k,t}((G_i),C,\cG)$. For each $j\in [m]$, let $\psi_{C_j}$ denote the canonical map from $G_{C_j}:=\prod_{i\in C_j} G_i$ to the quotient topological space $p_{C_j}(\cG)\backslash G_{C_j}$. Then for each $j$ we have $E_j(H)=\psi_{C_j}^{-1}(A_j)$, where $A_j$ is the Borel set $\psi_{C_j}(E_j(H))$.
\end{lemma}
Thus, the edge set of $H$ has the following form: $E(H) = \bigsqcup_{j\in [m]} \psi_{C_j}^{-1}(A_j)$. We call the sets $A_j$ the \emph{generators} of $H$.  When the groups $G_i$ are labelled copies of the same group $G$, we  write $H\in H_{k,t}(G,C,\cG)$. If we wish to specify the generators, we shall write $H=H_{k,t}(G,C,\cG,(A_j))$.\\
\indent The only thing there is to prove in Lemma \ref{lem:CayHypEdges} is that each generator $A_j$ is indeed a Borel set in $p_{C_j}(\cG)\backslash G_{C_j}$. This fact is not trivial, since a priori the $\sigma$-algebra of Borel sets on this quotient could be smaller than the $\sigma$-algebra obtained by pushing forward, via $\psi_{C_j}$, the Borel subsets of  $G_{C_j}$. In other words, we are using the following fact.

\begin{lemma}\label{lem:Haarquotient}
Let $G$ be a Hausdorff second-countable compact group, let $K$ be a closed subgroup of $G$, and let $\pi:G\to K\backslash G$ be the quotient map. Then for any $K$-invariant Borel set $E\subset G$, the set $\pi(E)$ is Borel.
\end{lemma}
This follows from results in descriptive set theory, for instance combining \cite[Theorem 12.17 and Corollary 15.2]{Kec}.\\

We now focus on \emph{abelian} groups, and for these we shall now relate Cayley $(t,m,k)$-graphs to systems of linear equations. From now on, given $M\in \Z^{r\times m}$ and an abelian group $G$, we shall write $(M,G)$ to refer to the system  $Mx=0$ with $x\in G^m$. Recall that our aim is to construct some invariant hypergraph $H$ such that Theorem \ref{thm:cag-rem-lem+} for $(M,G)$ can be deduced from the symmetry-preserving removal lemma for $H$.\\
\indent One of the simplest examples of such a construction, the idea of which can be traced back to Ruzsa and Szemer\'edi \cite{R&S}, concerns Schur's equation $x_1+x_2=x_3$. Let us revisit this example in order to motivate our general construction.

\begin{example}[Schur's equation, $M=(1\quad\,\,1\,\,\,-1)$]\label{ex:Schur} 
Consider the homomorphism $G^3\to G^3$ given by the following matrix:
\begin{equation}\label{eq:SchurPsi}
\Psi = \begin{pmatrix}
1  & -1  &  0 \\
0  &  1  &  -1\\
1  &  0  &  -1\\
\end{pmatrix}.
\end{equation}
This homomorphism has image equal to $\ker_G M$. Moreover, the row structure of $\Psi$ allows us to define a very convenient tripartite Cayley graph, given Borel sets $A_1,A_2,A_3\subset G$. Indeed, let $H$ be the $(3,3,2)$-graph with three vertex sets equal to $G$, with edge color classes $C_1=\{1,2\}$, $C_2=\{2,3\}$, $C_3=\{1,3\}$, and with $j$-th edge-set $E_j=\psi_{C_j}^{-1} A_j$, where the map $\psi_{C_j}: G^{C_j}\to G$ is given by the $j$-th row of $\Psi$. (Thus for instance $\psi_{\{1,2\}}$ takes a couple $(v(1),v(2))$ from the product of the first two vertex sets to $v(1)-v(2)$.) Letting $F$ be the triangle graph with vertices $1,2,3$, it can be checked easily that for each homomorphism $v=(v(1),v(2),v(3))$ of $F$ in $H$, the image $x=\Psi(v)$ is an element of $A_1\times A_2\times A_3 \cap \ker_G M$, and that we have in fact $\mu_{\ker_G M}(A_1\times A_2\times A_3 \cap \ker_G M)=\hd(F,H)$. Moreover, $H$ is a Cayley graph invariant under the 3-partite action of $\cG=\ker_G \Psi$, which   means here that each Borel set $E_j$ is a union of cosets of $\ker_G \psi_{C_j}$. Therefore, if $S_j$ is a set of small measure that is also a union of such cosets, then removing it from $E_j$ corresponds to removing a subset of small measure from $A_j$. We can thus establish Theorem \ref{thm:cag-rem-lem+} for $(M,G)$ using Lemma \ref{lem:simrem}.
\end{example}

In order to generalize the argument above, we shall now define a type of group homomorphism $\Psi$ that will enable us to associate a useful invariant hypergraph with a given system $(M,G)$. The definition uses the following notation.\\
\indent For a group $G$ and a subset $e$ of $[t]$, we denote by $\gamma_e$ the homomorphism embedding the direct power $G^e$ into $G^t$, defined by letting $\gamma_e(g')$ be the element $g$ such that $g(i)=g'(i)$ for $i\in e$ and $g(i)=0_G$ otherwise.\\
\indent Given any abelian groups $G_1,G_2$, and $m,t\in \N$, any homomorphism $\Psi: G_1^t \to G_2^m$ can be viewed as an $m \times t$ matrix of homomorphisms $G_1\to G_2$, namely for each $(j,k)\in [m]\times [t]$ the entry $\Psi_{j,k}$ is the homomorphism $p_j \circ \Psi \circ \gamma_k: G_1\to G_2$. We denote by $\psi_j$ the $j$-th row of this matrix, that is the homomorphism $\psi_j := p_j\circ \Psi= \sum_{k\in [t]} \Psi_{j,k}\circ p_k : G_1^t \to G_2$.\\
\indent We write $\supp \psi_j$ for the set of $k\in [t]$ such that $\Psi_{j,k}$ is not the 0-homomorphism $G_1\to \{0_{G_2}\}$. When $\supp \psi_j$ is a proper subset $C_j$ of $[t]$, we will often want to work with the  homomorphism $\psi_j\circ \gamma_{C_j}: G_1^{C_j}\to G_2$ rather than with $\psi_j: G_1^t\to G_2$. To simplify the notation, we shall denote $\psi_j\circ \gamma_{C_j}$ by  $\psi_{C_j}$.\\ 
\indent We can now give the main definition of this section.

\begin{defn}[Hypergraph representation]\label{defn:Cayley-rep}
Let $G$ be an abelian group, and let $M\in \Z^{r\times m}$. A  $(t,m,k)$\emph{-representation} of the system $(M,G)$ is a  homomorphism $\Psi : G_*^t \to G^m$, for some abelian group $G_*$, such that the following conditions hold:
\begin{enumerate}
\item There are distinct sets $C_1,C_2,\ldots, C_m \in  \binom{[t]}{k}$ such that $\forall\, j\in [m]$, $\supp \psi_j \subset C_j$.
\item  $\Psi(G_*^t)=\ker_G M$. 
\item For each $j\in [m]$, we have $p_{C_j}(\ker_{G_*} \Psi)=\ker_{G_*} \psi_{C_j}$.
\end{enumerate}
When $G$ is second-countable compact, we require that the same be true for $G_*$, and that $\Psi$ be continuous.
\end{defn}
A simple example is given by the matrix $\Psi$ in \eqref{eq:SchurPsi}, which gives a $(3,3,2)$-representation for Schur's equation on any abelian group $G$, where we can take $G_*=G$.\\
\indent The following proposition is the main result of this section.

\begin{proposition}\label{prop:keylink}
Let $M\in\Z^{r\times m}$, let $G$ be a second-countable compact abelian group, and suppose that the system $(M,G)$ has a $(t,m,k)$-representation for some positive integers $t,k$. Then Theorem \ref{thm:cag-rem-lem+} holds for $(M,G)$ with $\delta(\epsilon)=\delta(t,k,\epsilon)$, where $\delta(t,k,\epsilon)$ is given by Lemma \ref{lem:simrem}.
\end{proposition}

\begin{proof}
Let $\Psi$ be a $(t,m,k)$-representation for $(M,G)$. Fix $\epsilon>0$ and let $\delta=\delta(t,k,\epsilon)>0$ be such that Lemma \ref{lem:simrem} holds.\\
\indent Let $A_1,A_2,\ldots,A_m$ be Borel subsets of $G$ such that
\[
\mu_{\ker_G M}( A_1\times \cdots \times A_m \cap \ker_G M)\leq \delta.
\]
We may assume that each $A_j$ is a subset of $G^{(j)}=p_j(\ker_G M)$, since the part of $A_j$ outside the latter subgroup does not contribute to the above measure.\\
\indent Let $H= H_{k,t}(G_*,C,\cG,(A_j))$ be the Cayley $(t,m,k)$-graph given by $\Psi$, that is the hypergraph with vertex sets $V_i=G_*$, with $j$-th edge-color-class $C_j$, with $\cG=\ker_{G_*} \Psi$, and with generators $A_j$, $j\in [m]$. Let $F$ be the $k$-uniform hypergraph on $[t]$ with edges $C_1,C_2,\ldots,C_m$.\\
\indent Note the following fact concerning the quotient group $G_*^{C_j}/p_{C_j}(\cG)$ from Definition \ref{defn:Cay-hyp}: 
\begin{equation}\label{eq:labels}
\forall\,j\in [m],\quad G_*^{C_j}/p_{C_j}(\ker_{G_*}\Psi)\;\; \cong \;\; \psi_j(G_*^t) \;\; = \;\; p_j(\ker_G M) =:G^{(j)},
\end{equation}
this being an isomorphism of compact abelian groups. Indeed, by condition (iii) of Definition \ref{defn:Cayley-rep}, we have $G_*^{C_j}/p_{C_j}(\ker_{G_*}\Psi)=G_*^{C_j}/\ker_{G_*} \psi_{C_j}$. By the first isomorphism theorem, this is isomorphic  as a compact abelian group to $\psi_{C_j}\big(G_*^{C_j}\big)$. Since $\supp \psi_j\subset C_j$, we have $\psi_{C_j}\big(G_*^{C_j}\big)=\psi_j(G_*^t)$. By definition of $\psi_j$, the latter group is $p_j\circ \Psi(G_*^t)$, and by condition (ii) this is $p_j(\ker_G M)$.\\
\indent Now, since the map $\Psi$ is measure-preserving from $G_*^t$ onto $\ker_G M$ (as a continuous surjective homomorphism between compact abelian groups), we have
\begin{eqnarray*}
\hd(F,H) & =  & \int_{G_*^t} \prod_{j\in [m]} 1_{E_j}(p_{C_j}(g)) \ud g \;\;= \;\; \int_{G_*^t} \prod_{j\in [m]} 1_{A_j}(\psi_j(g)) \ud g\\
& =  & 
  \mu_{\Psi(G_*^t)} \Big( A_1\times \cdots \times A_m  \cap \Psi\big(G_*^t\big)\Big)\\
& = & \mu_{\ker_G M}( A_1\times \cdots \times A_m \cap \ker_G M) \leq \delta.
\end{eqnarray*}
By Lemma \ref{lem:simrem}, for each $j\in [m]$ there exists a Borel set $S_j \subset E_j(H)$, such that by removing $S_j$ from $E_j(H)$ for each $j\in [m]$ we obtain a $(t,m,k)$-graph $H'$ that is $F$-free and $\cG$-invariant. In particular, each set $S_j$ is invariant under the action of $p_{C_j}(\cG)=\ker_{G_*}\psi_{C_j}$, so by Lemma \ref{lem:CayHypEdges} there is a Borel  set $R_j\subset  G^{(j)}$ such that $E_j\setminus S_j=\psi_{C_j}^{-1}(A_j\setminus R_j)$. We also have $\mu_{G^{(j)}}(R_j)= \mu_{C_j}(S_j)\leq \epsilon$. 
Moreover, the set 
$\big(\prod_{j\in [m]} A_j\setminus R_j\big) \cap \ker_G M$ must be empty,  for if it contained some element $x=(x_1,\ldots,x_m)$ then there would be $g\in G_*^t$ such that $\Psi(g)=x$ and such that  $g_{C_j}\in \psi_{C_j}^{-1}(A_j\setminus R_j)=E_j\setminus S_j$ for each $j\in [m]$, contradicting the removal property of the sets $S_j$.
\end{proof}

We close this section by recording the deduction of Theorem \ref{thm:cag-Szem} from Theorem \ref{thm:cag-rem-lem}.

\begin{proof}[Proof of Theorem \ref{thm:cag-Szem}]
Suppose that $A\subset G$ is measurable with $\mu_G(A) \geq  \alpha>0$. Apply Theorem \ref{thm:cag-rem-lem} with $\epsilon = \alpha/2m$. Let $c=\delta(\epsilon)$ and suppose that $\mu_{\ker_G M}(A^m\cap \ker_G M)< c$. Then by Theorem \ref{thm:cag-rem-lem} there exists a measurable set $R\subset A$ of measure at most $\alpha/2$ such that $A\setminus R$ is $M$-free. However, $A\setminus R$ has measure at least $\alpha/2>0$, so it is non-empty, therefore it is not $M$-free (by invariance of $M$), a contradiction.
\end{proof}

\section{Finding a hypergraph representation for a given linear system} \label{section:FindRep}

Having established Proposition \ref{prop:keylink}, the proof of Theorem \ref{thm:cag-rem-lem+} is reduced to the following task: given an integer matrix $M$ with determinantal 1 and any abelian group $G$, show that the system $(M,G)$ admits a $(t,m,k)$-representation, with $t,m,k$ depending only on $M$. In fact, we shall need to complete this task only for  matrices that do not satisfy the following property.

\begin{defn}\label{defn:plain}
We say that $M\in \Z^{r\times m}$ is \emph{plain} if there exists $\ell\in [m]$ such that $p_\ell(\ker_G M)=\{0_G\}$ for every abelian group $G$.
\end{defn}
This notion is a special case of that of a `thin system' from \cite{KSVGR}. Examples include any square matrix $M\in \Z^{r\times r}$ with $d_r(M)=\det M =1$, since this has $\ker_G M=\{0_{G^r}\}$.

The following result allows us to restrict the above-mentioned task to non-plain matrices. Recall from the previous section the definition of the embedding homomorphism $\gamma_e : G^e \to G^m$ for a given $e\subset [m]$.
\begin{lemma}\label{lem:plaincase}
Let $M\in \Z^{r\times m}$ be a plain matrix satisfying $d_r(M)=1$. Then either Theorem \ref{thm:cag-rem-lem+} holds for $M$, or for some $s\in [r-1]$ there exists a matrix $M'\in \Z^{(r-s)\times (m-s)}$ that is not plain and such that, for some set $C\subset [m]$ of size $m-s$, the map $\gamma_C$ yields a \textup{(}measure-preserving\textup{)} isomorphism from $\ker_G M'$ to $\ker_G M$.
\end{lemma}
In other words, any element $x'\in\ker_G M'$ can be extended uniquely to an element $x\in \ker_{G}M$ by adding coordinates equal to $0_G$ with indices in $[m]\setminus C$ (i.e. corresponding with columns from $M$ missing in $M'$). As a consequence, if Theorem \ref{thm:cag-rem-lem+} holds for $M'$ then it holds for $M$. 
\begin{proof}
From Definition \ref{defn:plain} we have $p_\ell(\ker_\Q M)=\{0\}$ for some $\ell \in [m]$. We claim that then there exists a unimodular matrix $U\in \Z^{r\times r}$ such that the matrix $M_0=UM$ has first row equal\footnote{Incidentally, this claim also implies that in Definition \ref{defn:plain}, if  $d_r(M)=1$, then the case $G=\Q$ of the definition (i.e. $p_\ell(\ker_\Q M)=\{0\}$) implies the general case.} to the standard basis element $e_\ell$. To see this, note that $e_\ell$ must be in the row space over $\Q$ of the rows of $M$, that is $e_\ell \in \Span_\Q\{M_1,\ldots,M_r\}$. (Indeed, our assumption is that $\{e_\ell\}^\perp \supset \ker_\Q M$, so $e_\ell\in (\ker_\Q M)^\perp=\{M_1,\ldots,M_r\}^{\perp \perp}$.) Thus $e_\ell\in \Span_\Q\{M_1,\ldots,M_r\}\cap \Z^m$. But this set equals $\Span_\Z\{M_1,\ldots,M_r\}$ because $d_r(M)=1$, as can be seen using the Smith normal form $M= V (I_r|0) W$ (where $V\in \Z^{r\times r},W\in \Z^{m\times m}$ are unimodular, and $I_r$ denotes the identity matrix of order $r$). Thus we have $e_\ell= n_1 M_1+\cdots+n_r M_r$ where the $n_i$ are coprime integers. By \cite[Lemma 9]{KSVGR}, there exists a unimodular matrix $U\in \Z^{r\times r}$ with first row equal to $(n_1,n_2,\ldots, n_r)$. Thus we have $M_0=UM$ as claimed, and so $\ker_G M=\ker_G M_0$. Now, with the notation from Theorem \ref{thm:cag-rem-lem+}, if $0_G\notin A_\ell$, then Theorem \ref{thm:cag-rem-lem+} holds as there are no solutions $x$ with $x_i\in A_i$, for all $i\in[m]$. Otherwise, we remove the first row of  $M_0$ as well as the $\ell$-th column, obtaining a matrix $M'$. Note that the embedding homomorphism $\gamma_{[m]\setminus \{\ell\}}$  is a measure-preserving isomorphism $\ker_G M' \to \ker_G M$. If $M'$ is plain, we repeat the same procedure.\\
\indent This iteration must produce the desired matrix $M'$ before all the rows of $M_0$ are removed, for otherwise we would have that $0_{G^m}$ is the only solution and that $0_{G}\in A_j$ for each $j\in[m]$, which implies that $\mu_{\ker_G M}(\prod_j A_j \cap \ker_G M)=1$, contradicting the assumption in Theorem~\ref{thm:cag-rem-lem+}.
\end{proof}

Thus, our objective in this section is to prove the following result.

\begin{proposition}[Existence of a hypergraph representation]\label{prop:repexist}
Let $M\in \Z^{r\times m}$ with $d_r(M)=1$, and suppose that $M$ is not plain. Then there exist positive integers $t,k$ such that, for any abelian group $G$, the system $(M,G)$ has a $(t,m,k)$-representation.
\end{proposition}

The combination of this result with Proposition \ref{prop:keylink} (via Lemma \ref{lem:plaincase} if $M$ is plain) establishes Theorem \ref{thm:cag-rem-lem+}.

\begin{remark}\label{rem:coprime}
Note that if two matrices $M,M'\in \Z^{r\times m}$ satisfy $\ker_G M=\ker_G M'$, then a $(t,m,k)$-representation for $(M',G)$ is also a $(t,m,k)$-representation for $(M,G)$. In particular, for \emph{finite} abelian groups $G$, in Proposition \ref{prop:repexist} the assumption $d_r(M)=1$ can be relaxed to $\gcd(d_r(M),|G|)=1$ (and the same holds for the finite case of Theorem \ref{thm:cag-rem-lem}). Indeed, the Smith normal form of $M$ is then $U(D|0)V$ where $U\in \Z^{r\times r}, V\in \Z^{m\times m}$ are unimodular, and $D\in \Z^{r\times r}$ is a diagonal matrix with non-zero entries coprime  with $|G|$, so the endomorphism $D:G^r\to G^r$ is invertible. Therefore, letting $M'= (I_r|0)V$, we have that $\ker_G M=\ker_G M'$, whence $M$ has a $(t,m,k)$-representation if and only if $M'$ does.
\end{remark}

We shall prove Proposition \ref{prop:repexist} in several steps that constitute the subsections below. One of the main tools that we shall use is a notion of extension for integer matrices, which will enable us to replace the given matrix $M$ by a simpler one at each step of the argument. To define this notion of extension, we use the following notation. Given a set $J\subset [m']$ of size $m$, and a group $G$, recall that we denote by $p_J$ the coordinate projection $G^{m'}\to G^J$.  Instead of the image group $G^J$, we shall often want to work with the group $G^m$, isomorphic to $G^J$. To avoid a possibly confusing abuse of notation, we shall denote by $\pi_J$ the homomorphism $G^{m'}\to G^m$ that takes $(g_j)_{j\in [m']}$ to $(g_{\sigma_J(j)})_{j\in [m]}$, where $\sigma_J$ is the order-preserving bijection $[m]\to J$.

\begin{defn}[Matrix extension]\label{defn:sysext}
	Let $r'\geq r$, $m'\geq m$ and let $M\in \Z^{r\times m}, M'\in  \Z^{r' \times m'}$. We say that $M'$ is an \emph{extension} of $M$ if the following holds. There is a subset $J\subset [m']$ of size $m$ such that, for any abelian group $G$, the homomorphism $\pi_J:G^{m'}\to G^m$ restricts to an isomorphism $\ker_G M' \to \ker_G M$.
\end{defn}

Note that if $M$ and $M'$ both have full rank, then we must have $m'-r'=m-r$, since this is the dimension of their isomorphic kernels over $G=\Q$. The extensions that we shall consider will always be given by a matrix $M'$ having $M$ as a submatrix in such a way that $\pi_J$ has the required property.\\

\indent The key fact that makes extensions useful for us is that  they preserve the property of having a hypergraph representation, in the following sense.

\begin{lemma}\label{lem:key4ext}
Let $M\in \Z^{r\times m}$, and let $M'\in \Z^{r'\times m'}$ be an extension of $M$ with corresponding index set $J\in \binom{[m']}{m}$. Let $G$ be an abelian group, and suppose that $\Psi'$ is a $(t,m',k)$-representation for $(M',G)$. Then $\Psi:=\pi_J \circ \Psi' $ is a $(t,m,k)$-representation for $(M,G)$. Moreover if $\Psi'$ is given by an integer matrix, then so is $\Psi$.
\end{lemma}

\begin{proof}
First note that we can express the projection $\pi_J$ as left-multiplication by the $m\times m'$ integer matrix whose $j$-th row is  the vector with entry $\sigma_J(j)$ equal to 1 and all other entries 0, for each $j\in [m]$. Thus $\Psi$ is an $m\times t$ homomorphism matrix with $j$-th row equal to the $\sigma_J(j)$-th row of $\Psi'$, with support $C_j=C_{\sigma_J(j)}'$, where the latter is the support of the $\sigma_J(j)$-th row of $\Psi'$. In particular, the claim in the last sentence of the lemma is clear. Let us now check that the conditions of Definition \ref{defn:Cayley-rep} are satisfied.\\
\indent Condition (i) is inherited by $\Psi$ from $\Psi'$, since the $m$ rows of $\Psi$ form a subset of the $m'$ rows of $\Psi'$.\\
\indent Condition (ii) is also satisfied, indeed we have
\[
\Psi\big(G_*^{m'}\big)=\pi_J\big(\Psi'\big(G_*^{m'}\big)\big)=\pi_J(\ker_G M')=\ker_G M,
\]
where the last equality follows from Definition \ref{defn:sysext}.\\
\indent To check condition (iii), fix $j\in [m]$. Then, given $y\in  \ker_{G_*} \Psi\leq G_*^t$, we must have in particular the $j$-th coordinate of $\Psi(y)$ equal to 0, and this coordinate equals $\psi_{C_j}(p_{C_j}(y))$ by condition (i), whence $p_{C_j}(\ker_{G_*} \Psi)\subset \ker_{G_*} \psi_{C_j} $. To see the opposite containment, let $y'\in \ker_{G_*} \psi_{C_j}$. Since $C_j=C'_{\sigma_J(j)}$, we have $y'\in \ker_{G_*} \psi'_{C'_{\sigma_J(j)}}$, and since condition (iii) holds for $\Psi'$, there exists $y\in \ker_{G_*} \Psi'\subset \ker_{G_*} \Psi$ such that $p_{C'_{\sigma_J(j)}}(y)=p_{C_j}(y)=y'$, so we have indeed $\ker_{G_*} \psi_{C_j} \subset p_{C_j}(\ker_{G_*} \Psi)$.\\
\indent If $G,G_*$ are topological groups and $\Psi'$ is continuous, then so is $\Psi$.
\end{proof}

\subsection{A reduction to matrices of the form $(I_r|B)$}
Our first application of matrix extensions consists in showing that to establish Proposition \ref{prop:repexist} it suffices to prove it for matrices $M=(I_r|B)$. To that end we shall use the following result, the role of which is analogous to \cite[Lemma 10]{KSVGR}. 

\begin{lemma}\label{lem:firstext}
Let $M\in \Z^{r\times m}$ and suppose that $d_r(M)=1$ and that $M$ is not plain. Then $M$ has an extension $M'\in \Z^{m \times (2m-r)}$, with $J=[m]$ and $M'= U(I_m|B)$, where $U\in \Z^{m\times m}$ is unimodular and every row of $B$ is non-zero.
\end{lemma}

\begin{proof}
By \cite[Lemma 9]{KSVGR} there exists an $m\times m$ matrix 
$U=\left(\begin{array}{c} M \\ E \\ \end{array} \right)$ 
satisfying $\det (U)=d_r(M)=1$. We set $M'$ to be the following $m\times (2m-r)$ matrix:
\[
M'=
\begin{pmatrix}
M & 0 \\
E & I_{m-r}\\
\end{pmatrix}=(U|B_0).
\]
This is an extension of $M$ with $J=[m]$. Letting $B=U^{-1} B_0$, we have $U^{-1} M'= ( I_m | B)$ as required.\\
\indent Since $U$ is unimodular, we have $\ker_G M'=\ker_G (I_m|B)$ for any abelian group $G$. Therefore, if for some $j\in [m]$ the row $B_j$ of $B$ is 0, then
\[
p_j(\ker_G M')=p_j(\ker_G (I_m|B))=\{0_G\}.
\]
Then, since $\pi_J:\ker_G M'\to \ker_G M$ is an isomorphism, we must also have $p_j(\ker_G M)=\{0_G\}$, whence $M$ is plain.
\end{proof}

We can now reduce the proof of Proposition \ref{prop:repexist} to establishing the following result.

\begin{proposition}\label{prop:diagformrep}
Suppose that $M\in \Z^{r\times m}$ is of the form $M=(I_r|B)$, where $m\geq r+1$ and all rows of $B$ are non-zero. Then there exist positive integers $t,k$ such that, for any abelian group $G$, the system $(M,G)$ has a $(t,m,k)$-representation.
\end{proposition}

\begin{lemma}
Proposition \ref{prop:diagformrep} implies Proposition \ref{prop:repexist}.
\end{lemma}
\begin{proof}
Suppose that $M_0\in \Z^{r_0\times m_0}$ satisfies the assumptions in Proposition \ref{prop:repexist}. We have $m_0\geq r_0+1$, otherwise $M_0$ is plain. Applying Lemma \ref{lem:firstext}, we obtain a matrix $M'=U(I_r|B)\in \Z^{r\times m}$, with $r=m_0$, $m=2m_0-r_0\geq r+1$, and all rows of $B$ non-zero, such that $M'$ is an extension of $M_0$ with $J=[m_0]$. Let $M=(I_r|B)$, let $t,k$ be the integers given by Proposition \ref{prop:diagformrep}, and let $G$ be an abelian group. By Proposition \ref{prop:diagformrep} there is a $(t,m,k)$-representation $\Psi$ for $(M,G)$. Then, since $U$ is unimodular, we have $\ker_G M' = \ker_G M$, and so $\Psi$ is also a $(t,m,k)$-representation for $(M',G)$, just by Definition \ref{defn:Cayley-rep}.  Hence, by Lemma \ref{lem:key4ext}, the map $\pi_{[m_0]}\circ \Psi$ is a $(t,m_0,k)$-representation for the original system $(M_0,G)$.
\end{proof}

Our goal now is to prove Proposition \ref{prop:diagformrep}. To begin with, in the next subsection we deal with a special case consisting of what we call \emph{simple} matrices.

\subsection{Simple matrices}

Given a non-zero element $v\in \Z^m$, we denote by $\gcd(v)$ the greatest common divisor of the integers $v(i),i\in [m]$.

\begin{defn}
We say that a matrix $M \in \Z^{r\times m}$ of the form $(I_r|B)$ is \emph{simple} if $m\geq r+2$, and for each $i\in [r]$ the $i$-th row of $B$, denoted $B_i$, is non-zero and satisfies $\gcd(B_i)=1$.
\end{defn}
Our main result concerning these matrices is the following.

\begin{proposition}\label{prop:simplesystems}
Suppose that $M=(I_r|B)\in \Z^{r\times m}$ is simple. Then, for some positive integers $t,k$, there exists $\Psi\in \Z^{m\times t}$ such that, for any abelian group $G$, the homomorphism $\Psi:G^t\to G^m$ is a $(t,m,k)$-representation for the system $(M,G)$.
\end{proposition}

Given $M\in \Z^{r\times m}$, for each $j\in [m]$  let $M_{(j)}$ denote the square matrix formed by the columns of $M$ with indices $j-r\mod m,j-(r-1)\mod m,\ldots, j-1\mod m$.\\
\indent The main part of the proof of Proposition \ref{prop:simplesystems} consists in showing that any simple matrix has an extension with the particularly convenient property of being what we call a circular matrix.

\begin{defn}[Circular matrix]\label{def:circmatrix}
We say that a matrix $M\in \Z^{r\times m}$ is \emph{circular} if for each $j\in [m]$ the matrix $M_{(j)}$ is unimodular.
\end{defn}

\begin{proposition}\label{prop:simplext} Let $M\in \Z^{r\times m}$ be a simple matrix. Then there is an extension $M' \in \Z^{r'\times m'}$ such that $M'$ is circular.
\end{proposition}

Before we turn to the proof, let us motivate this proposition by briefly discussing circular matrices. One of the simplest examples of a circular matrix is the one corresponding to Schur's equation, that is $M = (1\;\;\;\;\; 1\;\;-1)$. In Example  \ref{ex:Schur} we saw that this has a nice representation, given in \eqref{eq:SchurPsi}, having the triangle as its corresponding graph $F$. Circular matrices are very convenient in that they provide simple generalizations of this construction, as shown by the following result.

\begin{lemma}\label{lem:circrep}
Suppose that $M\in \Z^{r\times m}$ is circular and that $m\geq r+2$. Then there exists $\Psi \in \Z^{m\times m}$ such that for any abelian group $G$, the homomorphism $\Psi:G^m\to G^m$ is an $(m,m,r+1)$-representation of $(M,G)$, with $C_j=\{j,j+1\mod m,\ldots, j+r\mod m\}$ for each $j\in [m]$.
\end{lemma}
Thus, the triangle graph corresponding to Schur's equation is generalized here to the `cyclic' $(r+1)$-uniform hypergraph on $[m]$ with edges $C_j$. Analogues of this construction have been used in previous works (though not in relation to hypergraph representations as defined here), specifically in \cite{C, KSVGR}. In particular, Definition \ref{def:circmatrix} is an analogue of the notion of `$n$-circular matrix' used in \cite{KSVGR}. 
\begin{proof}[Proof of Lemma \ref{lem:circrep}]
We construct $\Psi$ as follows: the $j$-th column $\Psi^j$ is an element of $\Z^m$ lying in $\ker_\Q M$, with support inside $\{j-r,j-(r-1),\ldots, j-1,j\}$ (subtractions mod $m$), and with $j$-th entry equal to $-1$. More precisely, let $y= M_{(j)}^{-1}  M^j \in \Z^r$; then $M_{(j)} y = M^j$. We then define $\Psi^j$ by $\Psi^j(i)=y(i)$ for $i\in \{j-r,j-(r-1),\ldots, j-1\}$,  $\Psi^j(j)=-1$, and $\Psi^j(i)=0$ otherwise. Note that we have indeed $\Psi^j\in \Z^m$ and $M \Psi^j = 0$. Note also that the resulting matrix $\Psi$ has row $j$ with support indeed contained in the set $C_j=j+[0,r] \mod m$, and that these sets $C_j$ are distinct since $m\geq r+2$. Hence, condition (i) from Definition \ref{defn:Cayley-rep} is satisfied with $k=r+1$.\\
\indent Let us check condition (ii), i.e. that $\Psi(G^m)= \ker_G M$. Since $M\Psi=0$, we clearly have $\Psi(G^m)\subset \ker_G M$. To see equality, fix any $x\in \ker_G M$. 
Observe that $x$ is uniquely determined by any sequence of $m-r$ consecutive coordinates mod $m$, because the submatrix formed by the remaining $r$ columns of $M$, being unimodular, gives a bijection on $G^r$. Hence, if we find $y\in G^m$ such that $\Psi(y)$ agrees with $x$ on such a sequence of $m-r$ coordinates, then this together with the fact that $M\Psi(y)=0$ will imply that $x=\Psi(y)\in \Psi(G^m)$. Now note that the top-left square submatrix of $\Psi$ of order $m-r$ is upper triangular with entries $-1$ in the diagonal, so we  can indeed find the desired element $y$.\\
\indent To check condition (iii), let $\cG=\ker_G \Psi$. We have to check that for each $j\in [m[$, the map $\Psi_{C_j}$ (the restriction of the $j$-th row of $\Psi$ to $G^{C_j}$) satisfies $\ker \psi_{C_j}=p_{C_j}(\cG)$. Clearly $p_{C_j}(\cG)\subset \ker\psi_{C_j}$, since if $\Psi(y)=0$ then in particular the $j$-th entry, which equals $\psi_{C_j}\big(p_{C_j} (y)\big)$, is 0. To see the opposite containment, suppose that $y'\in G^{C_j}$ satisfies $\psi_{C_j}(y')=0$. We want to show that there exists $y\in \cG$ such that $p_{C_j}(y)=y'$. Using the row structure of $\Psi$, we can find successively elements $y_{j-1},y_{j-2},\ldots, y_{j-(m-r-1)}\in G$ such that for each $\ell\in \{ j,j-1,\ldots, j-(m-r-1)\}$ we have $\psi_{C_{\ell}} (y_{\ell},y_{\ell+1},\ldots, y_{\ell+r})=0$. We use these elements $y_\ell$ to extend $y'$ to an element $y\in G^m$, defined by $p_{C_j}(y)=y'$ and $p_\ell(y)=y_\ell$ for each $\ell\in [m]\setminus C_j$. By construction, $\Psi y$ has $m-r$ coordinates equal to 0.  Since 0 and $\Psi y$ are both in $\ker_G M$ and agree on these $m-r$ coordinates, by the observation in the previous paragraph we must have $\Psi y = 0$, so $y\in \cG$ as required.\\
\indent If $G$ is a topological group then $\Psi$ is clearly continuous. We have thus shown that all the conditions in Definition \ref{defn:Cayley-rep} are satisfied.
\end{proof}
Let us now turn to establishing Proposition \ref{prop:simplext}. The proof is an adaptation of an argument from \cite{KSVGR}. Given a matrix $B$, we shall denote by $B_{[i_1,i_2]}$ the submatrix of $B$ formed by consecutive rows with indices $i_1,i_1+1,\ldots, i_2$. First we adapt \cite[Lemma 11]{KSVGR}, to obtain the following.

\begin{lemma}\label{lem:adapt11}
Let $B\in \Z^{r\times r}$ be a unimodular matrix. Then for some integer $s=O_B(1)$, there exist integer matrices $S,T$ such that the $s\times r$ matrix
\begin{equation}\label{eq:adapt11}
\overline{B}=\left(\begin{array}{c} I_r\\ S\\B\\T\\I_r\end{array}\right)
\end{equation}
satisfies the following property: for each $i\in [1,s-(r-1)]$, the $r\times r$ submatrix $\overline{B}_{[i,i+r-1]}$ is unimodular.
\end{lemma}
In other words, each submatrix of $\overline{B}$ formed by $r$ consecutive rows is unimodular.

\begin{proof} We shall say that an integer matrix with $r$ columns is \emph{good} if each of its square submatrices formed by $r$ consecutive rows is unimodular.

We first claim that there exists a matrix $T$ such that the matrix $L=\left(\begin{array}{c}
			B \\
			T\\
			I_r\\
		\end{array}\right)$ is good and has $O_B(1)$ rows. (The upper part of $\overline{B}$ will be found analogously.)\\
\indent	This holds for $r=1$ since we can set $L=\left(\begin{array}{c} \pm 1\\ 1\end{array}\right)$. For $r>1$, we can suppose by induction that the claim holds for $r-1$. The matrix $L$ is constructed by repeatedly adding an appropriate new row at the bottom of $B$ while ensuring that the new bottom $r\times r$ submatrix is unimodular. The idea is that each new row essentially captures a step in an application of Euclid's algorithm to the entries in the first column of $B$.\\
\indent Thus we first form the matrix
$\left(\begin{array}{c}
			B_2 \\
			B_3\\
			\vdots\\
			B_r\\
			\sum_{i=1}^r \lambda_i B_i
			\end{array}\right),$ 
where $\lambda_1\in\{-1,1\}$, and such that $|\lambda_1 B_{1,1}+\sum_{i=2}^r \lambda_i B_{i,1}|$ is minimized.  This coefficient $\lambda_1$ having magnitude 1 ensures that the new row is the result of an elementary row operation on $B$, so that the above matrix is still unimodular. If $|B_{1,1}|=\max_{i\in [r]} |B_{i,1}|>1$, then we can find $\{\lambda_i: i\in [r]\}$ with $|\lambda_1|= 1$ such that $|\lambda_1 B_{1,1}+\sum_{i=2}^r \lambda_i B_{i,1}|<|B_{1,1}|$. Otherwise, note that we can certainly make the left side here at most $|B_{1,1}|$. It follows that after repeating this addition of a new row at most $r$ times, we have decreased the greatest magnitude of the entries in the first column (provided that this magnitude was greater than 1). We then iterate this process, denoting by $T_i$ the $i$-th new bottom row. By Euclid's algorithm, after $\ell=O_B(1)$ steps, we obtain a good matrix $\left( \begin{array}{c} B\\ T' \end{array}\right)$ where 
$
T'=\left( \begin{array}{c} T_1\\T_2 \\ \vdots \\ T_\ell \end{array}\right) 
$ 
and the first entry of $T_\ell$ is $1$ (the entries in the first column of $B$ are coprime by assumption). Now we can carry out $r-1$ further steps consisting in subtracting integer multiples of $T_\ell$ from previous rows, to obtain the $r\times r$ matrix $T'=\left( \begin{array}{cc}  1 & \ast\\0&B'\end{array}\right)$,  with top row $T_\ell$, and with $B'\in\Z^{(r-1)\times (r-1)}$ being unimodular.\\
\indent Now we apply the induction hypothesis to $B'$, obtaining an $s' \times (r-1)$ good matrix $
L'=\left( \begin{array}{c} B' \\ T''\\ I_{r-1}\\ \end{array}\right)$, where $s'=O_{B'}(1)=O_B(1)$. We then add to $L'$ a first column of zeros, and we insert in the resulting matrix the row $(1,0,\ldots ,0)$ of length $r$ between the  positions $j(r-1)$ and $j(r-1)+1$, for each $j\in [1,s'-1]$. The resulting matrix $L''$ has $O_B(1)$ rows and, by construction, the following matrix is good:
$L=\left( \begin{array}{c} B \\ L'' \end{array}\right)=\left(\begin{array}{c} B \\ T\\ I_r \end{array}\right)$, 
for some matrix $T$ with $O_B(1)$ rows.\\
\indent The proof is completed by a similar argument adding top  rows to $B$, yielding the desired matrix $S$.	
\end{proof}
To complete the proof of Proposition \ref{prop:simplext}, we adapt the argument from \cite[Lemma 12]{KSVGR}, using Lemma \ref{lem:adapt11}  instead of \cite[Lemma 11]{KSVGR}.

\begin{proof}[Proof of Proposition \ref{prop:simplext}]
Since each row $B_i$ of $B$ satisfies $\gcd(B_i)=1$, by \cite[Lemma 9]{KSVGR} there exists an $(m-r)\times (m-r)$ unimodular matrix $U_i$ with top row equal to $B_i$. Applying Lemma \ref{lem:adapt11} to each such matrix $U_i$ we obtain $\overline{U_i}$ as given by \eqref{eq:adapt11}. We then form the following $r'\times (m-r)$ matrix, where $r' =O_B(1)$: 
$B'=
\begin{pmatrix}
\overline{U_1} \\
\overline{U_2} \\
\vdots \\
\overline{U_r}
\end{pmatrix}$.
Let $M'=(I_{r'}|B')$, and note that $M'=\left( I_{r'}\left| \begin{array}{c}I_{m-r} \\
X \\
I_{m-r}\end{array}\right.\right)$, for some matrix $X$.

We claim that $M'$ is circular. To see this, let $M'_{(i)}$ denote the square submatrix formed by $r'$ consecutive columns of $M'$ in the circular order, starting with the $i$-th column. Then, for the first $m-r$ and last $m-r$ values of $i\in [m']$ it is clear that $M'_{(i)}$ is unimodular; for example, for the first $m-r$  values, $M'_{(i)}$ is unimodular because its columns form a circular permutation of the columns of a lower triangular matrix with diagonal entries equal to 1. For $i\in (m-r,m'-(m-r)]=(m-r,r']$, note that $\det M'_{(i)} = \pm \det B'_{[i-(m-r),i-1]}=\pm 1$ so $M'_{(i)}$ is indeed unimodular.\\
\indent To complete the proof, let us specify the index set $J\subset [m']$ of size $m$ showing that $M'$ is an extension of $M$. Let $J_1\subset [r']$ be the set of size $r$ containing the subscript of each row of $B'$ that is the first row of a submatrix $U_i$ (recall that $U_i$ is a submatrix of $\overline{U_i}$). Since this first row is $B_i$ by construction, we thus have that the order-preserving bijection $\sigma_{J_1}:[r]\to J_1$ satisfies $B_i=B'_{\sigma_{J_1}(i)}$. We set $J=J_1\cup [r'+1,m'] \subset [m']$. From the structure of $M'$, it then follows that the homomorphism $\pi_J$ restricted to $\ker_G M'$ gives an isomorphism $\ker_G M'\to \ker_G M$, as required. Indeed, the submatrix of $M'$ formed by the rows indexed by $J_1$ is equal to $M$ (up to relabelling rows and columns) so if $x'\in \ker_G M'$ then $\pi_J(x')\in \ker_G M$; moreover,  given $x\in \ker_G M$, the element of $G^J$ with $j$-th coordinate $x(\sigma_J^{-1}(j))$ is in the kernel of the above submatrix, and using the structure of $M'$ we then extend this element uniquely to an element  $x'\in \ker_G M'$ such that $\pi_J(x')=x$ (since an element  $x'\in \ker_G M'$ is uniquely determined by its last $m'-r'=m-r$ coordinates).
\end{proof}

Let us finally combine the above ingredients to obtain the main result of this subsection.

\begin{proof}[Proof of Proposition \ref{prop:simplesystems}]
By Proposition \ref{prop:simplext} there is a circular matrix $M'=(I_{r'}|B')\in \Z^{r'\times m'}$ extending $M$. By Lemma \ref{lem:circrep} we have an $(m',m',r'+1)$-representation for $M'$ given by an integer matrix $\Psi'$. Hence, by Lemma \ref{lem:key4ext}, we obtain a $(t,m,k)$-representation $\Psi=\pi_J\circ \Psi'$ for $M$, also given by an integer matrix, with $t=m'$ and $k=r'+1$.
\end{proof}

\begin{remark}
Note that Proposition \ref{prop:simplesystems} establishes Proposition \ref{prop:diagformrep} for simple matrices in the strong sense that the conditions in Definition \ref{defn:Cayley-rep} of hypergraph representability  are satisfied with $G_*=G$ and $\Psi$ being just an integer matrix (rather than a more general homomorphism matrix). In the next subsection, the full generality of Definition \ref{defn:Cayley-rep} will be used to handle all remaining matrices of the form $(I_r|B)$. For this purpose, instead of matrix extensions, we shall use a different construction.
\end{remark}

\subsection{General matrices of the form $(I_r|B)$.}

In this subsection we complete the proof of Proposition \ref{prop:diagformrep}, by using Proposition  \ref{prop:simplesystems} to construct a hypergraph representation for any system $(M,G)$ with $M\in \Z^{r\times m}$ a non-plain matrix of the form $(I_r|B)$. This will establish Proposition \ref{prop:repexist}, and thereby we shall have completed the proof of Theorem \ref{thm:cag-rem-lem+}.\\
\indent Let us first explain briefly the main difficulty, and in particular why the rest of the argument from the finite setting of \cite{KSVGR} does not work in our setting of general compact abelian groups.\\
\indent We want to find a hypergraph representation for any given system with a non-simple $r\times m$ matrix $(I_r|B)$. There are two cases to treat:  in the main case we have $m\geq r+2$ and for some row $B_i$ of $B$ we have $\gcd(B_i)=s>1$; in the second case we have $m=r+1$.\\
\indent In \cite{KSVGR}, an analogue of the main case is addressed using a  notion of `system extension', which differs from the matrix extensions  used in this paper. In particular, the extensions in \cite{KSVGR} allow one to multiply group elements by $s$ and thus reduce the task to the case of simple matrices; see for instance the proof of \cite[Lemma 10]{KSVGR}. This multiplication is allowed in the setting of finite abelian groups because it does not increase the measure of sets, which is important to ensure that the measures of the removal sets are kept small. Our general setting includes groups in which multiplication by an integer may increase measures (e.g. the circle group), so we cannot use this argument. Instead, we shall construct a certain `covering' of the original kernel $\ker_G M$ by kernels of systems given by \emph{simple} matrices associated with the original matrix.\footnote{This construction plays a role, relative to Definition \ref{defn:Cayley-rep}, somewhat analogous to the role played by the blowup construction described in \cite{vKAM} relative to the arguments in \cite{KSVGR}.} Proposition  \ref{prop:simplesystems} gives us a representation for each of these simple matrices, and we shall then combine these to obtain a representation for the original system. This will address the first case stated above.\\
\indent The second case will then be simpler to handle and will be treated at the end of this section.

Before we go into the details of the main case, let us briefly illustrate the idea of the argument.

\begin{example}
Consider a system $(M,G)$ with $M=(1\;\; 2 \;\; 2)$, a non-simple matrix of the form $(I_1|B)$. Consider then the following two systems: $(M^{(0)},G_0)$ with $M^{(0)}=(1 \;\; 1 \;\; 1)$ and $G_0=G$, and $(M^{(1)},G_1)$ with 
 $M^{(1)}=(1 \;\; 2 \;\; 2 \;\; 1)$ and $G_1$ the subgroup of $G$ consisting of the preimages of $0_G$ under multiplication by $2$. Note that $M^{(0)}$ and $M^{(1)}$ are simple matrices, so we have representations given by matrices $\Psi^{(0)}\in \Z^{3\times 3}$ and $\Psi^{(1)}\in \Z^{4\times t}$ for $(M^{(0)},G_0)$, $(M^{(1)},G_1)$ respectively.

Let us denote an element of $\ker_{G_0} M^{(0)}$ by $\mathbf{x}^{(0)}=(x_{0,1},x_{0,2},x_{0,3})$ and similarly an element of $\ker_{G_1} M^{(1)}$ by $\mathbf{x}^{(1)}=(x_{1,1},x_{1,2},x_{1,3},x_{1,4})$. Consider the following map:
\begin{align}
	K\;:\;\ker_{G_0} M^{(0)}\times\ker_{G_1} M^{(1)} &\to \ker_G M \nonumber \\
(\mathbf{x}^{(0)},\mathbf{x}^{(1)}) & \mapsto\;\mathbf{x} = (2 (x_{0,1}+x_{1,1}),x_{0,2}+x_{1,2},x_{0,3}+x_{1,3}). \nonumber
\end{align}
This map is surjective. Moreover, the preimages of any solution $\mathbf{x}\in \ker_G M$ have convenient `covering' properties when fixing any given coordinate $\mathbf{x}(j)$ (see part (iii') of Definition \ref{defn:Cayley-rep-var}), properties which are obtained essentially by using the component $x_{1,4}$ as a free variable. These properties are then used to construct a representation $\Psi$, consisting essentially in a $3\times (3+t)$ matrix in which the left $3\times 3$ submatrix is given by $\Psi^{(0)}$ and the right submatrix is given by $\Psi^{(1)}$ without the last row.
\end{example}

To define our construction formally, we shall use the following simple fact.

\begin{lemma}\label{lem:blowup}
Let $M\in \Z^{r\times m}$, let $G$ be an abelian group, and suppose that $\Psi':{G_*'}^t \to G^m$ is a $(t,m,k)$-representation for $(M,G)$. Let $G_*$ be another abelian group and let $\pi:G_*\to G_*'$ be a surjective homomorphism \textup{(}if $G$ is compact second countable then we assume that the same holds for $G_*$, and that $\pi$ is continuous\textup{)}. Then letting $\pi^t$ denote the homomorphism $G_*^t\to {G'_*}^t$ mapping $g=(g_1,\ldots,g_t)$ to $(\pi(g_1),\ldots,\pi(g_t))$, we have that $\Psi:=\Psi'\circ \pi^t:G_*^t\to G^m$ is also a $(t,m,k)$-representation for $(M,G)$.
\end{lemma}

\begin{proof}
Condition (i) from Definition \ref{defn:Cayley-rep} holds clearly for $\Psi$.\\
\indent Condition (ii) is also clear: by assumption we have $\ker_G M= \Psi'({G_*'}^t)$ and this equals $\Psi'(\pi^t(G_*^t))$ by surjectivity of $\pi$.\\
\indent For condition (iii), fix any $j\in [m]$ and note that we certainly have  $p_{C_j}(\ker_{G_*}\Psi)\subset \ker_{G_*} \psi_{C_j}$. To see the opposite containment, suppose that $g'\in G_*^{C_j}$ satisfies $\psi_{C_j}(g')=0_G$ and note that since $\psi_{C_j}(g')=\psi'_{C_j}(\pi^{C_j}(g'))$ and condition (iii) holds for $\Psi'$, there exists $g_0'\in \ker_{G_*'}\Psi'$ such that $p_{C_j}(g_0')=\pi^{C_j}(g')$. By surjectivity of $\pi$ there exists $g_0\in G_*^t$ such that $\pi^t(g_0)=g_0'$. We may not have $p_{C_j}(g_0)=g'$, but we do have this equality mod $\pi^{C_j}$, that is we have
\[
\pi^{C_j}(p_{C_j}(g_0))=p_{C_j}(\pi^t(g_0))=p_{C_j}(g_0')=\pi^{C_j}(g').
\]
Hence there exists $h'\in \ker_{G_*} \pi^{C_j}$ such that $p_{C_j}(g_0)=g'+h'$. Adding coordinates equal to $0_{G_*}$ to $h'$, we obtain $h\in G_*^t$ such that $p_{C_j}(h)=h'$ and $\pi^t(h)=0$. Letting $g=g_0-h$, we have $\Psi(g)=\Psi'(g_0')-\Psi'(\pi^t(h))=0_{G^m}$ and $p_{C_j}(g)=g'$. Hence condition (iii) holds.\\
\indent Finally, if $\Psi'$ is continuous on the compact abelian group ${G_*'}^t$, then $\Psi$ is also continuous on the compact abelian group $G_*^t$.
\end{proof}

Let us now describe the construction for the main case in detail.\\
Let $M=(I_r|B)\in \Z^{r\times m}$ be the given matrix with $m\geq r+2$ and some row $B_i$ satisfying $\gcd(B_i)>1$. Let $B'$ denote the matrix obtained from $B$ by dividing, for each $i\in [r]$, each coordinate of $B_i$ by $\gcd(B_i)$.\\
\indent Let $M^{(0)}$ be the simple matrix $M^{(0)}=(I_r| B')$, and for each $i\in [r]$ let $M^{(i)}$ be the following $r \times (m+1)$ simple matrix
\[
M^{(i)} =
\begin{pmatrix}
I_{i-1}         & 0  &   0       &   B_{[1,i-1]}'   &  0\\
\mathbf{0} & 1  &   0        & B_i  &   1\\
\mathbf{0} & 0  & I_{r-i}   &  B_{[i+1,r]}' &   0\\
\end{pmatrix},
\]
where $B_{[i_1,i_2]}'$ denotes the submatrix of $B'$ formed by rows $i_1,i_1+1,\ldots, i_2$.\\
\indent Let $G$ be an arbitrary abelian group. For each $i\in [r]$ let $G_i$ denote the preimage of $0_G$ under multiplication by $\gcd(B_i)$ (thus $G_i\leq G$), and let $G_0=G$.\\
\indent Let $G_*= G_0 \times G_1\times \cdots \times G_r$. (Note that $G_*$ is compact second countable if $G$ is.) For each $i\in [0,r]$, let $\pi_i$ be the projection homomorphism  $G_*\to G_i$, and let ${\Psi'}^{(i)}$ be the matrix in $\Z^{m_i\times t_i}$ given by Proposition \ref{prop:simplesystems}, thus ${\Psi'}^{(i)}: G_i^{t_i}\to G_i^{m_i}$ is a $(t_i,m_i,k_i)$-representation for the system $(M^{(i)},G_i)$. Then, by Lemma \ref{lem:blowup}, the homomorphism
\[
\Psi^{(i)}:= {\Psi'}^{(i)}\circ \pi_i^{t_i} :  G_*^{t_i}\to G_i^{m_i}
\]
is also a $(t_i,m_i,k_i)$-representation for the system $(M^{(i)},G_i)$.\\
\indent Let $t = t_0+t_1+\cdots+t_r$. We shall now combine these representations $\Psi^{(i)}$ to define a map $\Psi:G_*^t\to G^m$.\\
\indent Each element $g\in G_*^t$ may be written in the form $g=\Big( g^{(0)},g^{(1)},\ldots, g^{(r)}\Big)$, where $g^{(i)}\in G_*^{t_i}$ for each $i\in [0,r]$. We then define the homomorphism 
\begin{eqnarray}\label{eq:PhiHom}
\Phi: G_*^t & \to & G_0^m\times G_1^{m+1} \times\cdots \times G_r^{m+1}\leq G^{m+r(m+1)}  \\
g & \mapsto & \Big(\Psi^{(0)}\big(g^{(0)}\big),\,\Psi^{(1)}\big(g^{(1)}\big),\ldots,\; \Psi^{(r)}\big(g^{(r)}\big)\Big). \nonumber
\end{eqnarray}
Note that $\Phi$ can be viewed as an $(m+r(m+1)) \times t$ matrix of homomorphisms, where the top-left $m\times t_0$ submatrix is $\Psi^{(0)}$, the submatrix on the next $m+1$ rows and $t_1$ columns is $\Psi^{(1)}$, and so on, and every other entry is the zero homomorphism.\\
\indent By condition (ii) from Definition \ref{defn:Cayley-rep} for each $\Psi^{(i)}$, we have
\[
\Phi\big(G_*^t\big)=\ker_{G_0} M^{(0)}\times \ker_{G_1} M^{(1)}\times \cdots \times \ker_{G_r} M^{(r)}.
\]
We shall denote an element of this group by $\mathbf{x}=\big(\mathbf{x}^{(0)},\mathbf{x}^{(1)},\ldots, \mathbf{x}^{(r)}\big)$, where $\mathbf{x}^{(i)}\in \ker_{G_i} M^{(i)}$.\\
\indent To complete the definition of $\Psi$, we shall now compose $\Phi$ with another homomorphism, denoted $K$, which will combine the entries of $\Phi(g)$ appropriately to produce an element of $\ker_G M$.\\
\indent We define
\begin{equation}
 K := P \circ \Sigma \;: \; \prod_{i=0}^r \ker_{G_i} M^{(i)} \;  \to \; G^m,
\end{equation}
where $\Sigma$ is the following coordinate-summation map:
\[
\begin{array}{cccl}
\Sigma:& \prod_{i=0}^r \ker_{G_i} M^{(i)} & \to &G^m \vspace{0.25cm}\\
&\mathbf{x}=\big(\mathbf{x}^{(0)},\mathbf{x}^{(1)},\ldots, \mathbf{x}^{(r)}\big) & \mapsto & \left(\sum_{i=0}^r \mathbf{x}^{(i)}(j)\right)_{j\in [m]},
\end{array}
\]
and where $P$ is the following coordinate-multiplication map:
\[
\begin{array}{cccl}
P: & G^m & \to & G^m \vspace{0.25cm}\\
& y & \mapsto & \big( \;\gcd(B_1) \,y_1,\; \gcd(B_2)\,y_2,\,\ldots,\, \gcd(B_r)\, y_r,\; y_{r+1},\ldots,\; y_m \;\big).
\end{array}
\]
Note that $\Sigma$ (and hence $K$) ignores the $(m+1)$-st component of $\mathbf{x}^{(i)}$ for each $i\in [r]$.
\begin{defn}\label{defn:genrep}
Let $M=(I_r|B)\in \Z^{r\times m}$ with $m\geq r+2$, let $G$ be an abelian group, and suppose that $M$ is not simple. Then we define the following homomorphism:
\begin{eqnarray}\label{eq:genrep}
\Psi :  G_*^t & \to & G^m \nonumber\\
 g & \mapsto & K \circ \Phi (g),
\end{eqnarray}
where $G_*,K,\Phi,t$ are as defined above.
\end{defn}
One can view $\Psi$ as an $m\times t$ matrix of homomorphisms, with $r+1$ submatrices formed by sets of consecutive columns, where the $i$-th submatrix is formed by $t_i$ such columns and is equal to $P\circ \Psi^{(i)}$, for each $i\in [0,r]$.

\begin{proposition}\label{prop:genrep}
Let $M=(I_r|B)\in \Z^{r\times m}$ with $m\geq r+2$ and let $G$ be an abelian group. Then the homomorphism $\Psi$ in \eqref{eq:genrep} is a $(t,m,k)$-representation for $(M,G)$, with $k=k_0+k_1+\cdots+k_r$.
\end{proposition}

Thus, the coordinates of an element of $G_*^t$ with indices in $\big[\sum_{j=0}^{i-1}t_j+1\;,\;\sum_{j=0}^{i}t_j\big]$ are used by $\Psi^{(i)}$ to represent $M^{(i)}$, and these representations are then combined by $K$ to make $\Psi$ a representation for $(M,G)$.\\
\indent To prove Proposition \ref{prop:genrep}, we first record an equivalent definition of a hypergraph representation, where condition (iii) from Definition \ref{defn:Cayley-rep} is replaced with a variant that is convenient for our arguments below. 

\begin{defn}[Hypergraph representation, equivalent formulation]\label{defn:Cayley-rep-var}
Let $M\in \Z^{r\times m}$ and let $G$ be an abelian group. A  $(t,m,k)$\emph{-representation} of the system $(M,G)$ is a  homomorphism $\Psi : G_*^t \to G^m$, for some abelian group $G_*$, such that the following conditions hold:
\begin{enumerate}
\item There are distinct sets $C_1,C_2,\ldots, C_m \in  \binom{[t]}{k}$ such that $\forall\, j\in [m]$, $\supp \psi_j \subset C_j$.
\item  $\Psi(G_*^t)=\ker_G M$. 
\item[(iii')] For every $j\in [m]$, for every $x=(x_1,\ldots, x_m) \in \ker_G M$, if $g'\in G_*^{C_j}$ satisfies $\psi_{C_j}(g')=x_j$, then there exists $g\in G_*^t$ such that $\Psi(g)=x$ and $p_{C_j}(g)=g'$. 
\end{enumerate}
When $G$ is a \emph{compact} abelian topological group, we require that $G_*$ also be compact, and that $\Psi$  be continuous.
\end{defn}

\begin{lemma}\label{lem:equiv-def}
Definitions \ref{defn:Cayley-rep} and \ref{defn:Cayley-rep-var} are equivalent.
\end{lemma}
\begin{proof}
Let us recall condition (iii) from Definition \ref{defn:Cayley-rep}:

(iii) For each $j\in [m]$, we have $p_{C_j}(\ker_{G_*} \Psi)=\ker_{G_*} \psi_{C_j}$.

To see that (iii') implies (iii), note that by definition we have $p_{C_j}(\ker_{G_*} \Psi )\subset \ker_{G_*} \psi_{C_j}$, and that the opposite containment also holds, by (iii') applied with $x=0_{G^m}$.\\
\indent To see that (iii) implies (iii'), suppose that $x\in \ker_G M$ and $g'\in G_*^{C_j}$ satisfy $\psi_{C_j}(g')=x_j$. Note that by condition (ii) we have $x=\Psi(g_0)$ for some $g_0\in G_*^t$. Then $\psi_{C_j}(g')=\psi_{C_j}\big(p_{C_j}(g_0)\big)=x_j$, so $g' - p_{C_j}(g_0) \in \ker_{G_*} \psi_{C_j}$. Therefore, by (iii), there exists $g_1\in \ker_{G_*} \Psi$ such that $p_{C_j}(g_1)=g' - p_{C_j}(g_0)$. Letting $g = g_0+g_1$, we thus have $p_{C_j}(g)=g'$ and $\Psi(g)=\Psi(g_0)=x$, so (iii') holds.
\end{proof}
We now turn to the proof of Proposition \ref{prop:genrep}. A central fact that we shall use is that $K$ ignores the $(m+1)$-st coordinates of each ${\mathbf{x}}^{(i)}$, $i\in[r]$, indeed this provides the additional degrees of freedom sufficient for establishing the required properties of $\Psi$.\\
\indent Let $C_{j,i}\subset [t_i]$ denote the support of the $j$-th row of $\Psi^{(i)}$. Then the support of the $j$-th row of $\Psi$ is
\[
C_j\;=\;C_j(0)\;\sqcup\; C_j(1)\; \sqcup\; \ldots\; \sqcup\; C_j(r)\;\subset [t],
\]
where $C_j(i)$ is the shifted set $(t_0+t_1+\cdots+t_{i-1})+C_{j,i} \subset [t]$.\\
\indent It is then clear that condition (i) from Definition \ref{defn:Cayley-rep-var} is inherited by $\Psi$ from the $\Psi^{(i)}$.

We prove the other two conditions separately.

\begin{lemma}\label{lem:cond2final}
The map $\Psi$ in \eqref{eq:genrep} satisfies $\Psi(G_*^t)=  \ker_G M$.
\end{lemma}
\begin{proof}
We first check that $\Psi(G_*^t)\subset  \ker_G M$. We have to ensure that for any $g\in G_*^t$, for each $j\in [r]$ we have $M_j(\Psi(g))=0$. Note that $\Psi(g)=K(\mathbf{x})$ for  $\mathbf{x}=\Phi(g)\in   \prod_{i=0}^r \ker_{G_i} M^{(i)} $. Let ${\mathbf{x}'}^{(i)}=p_{[m]}(\mathbf{x}^{(i)})$, for each $i\in [r]$. Observe that, letting $P'$ denote the map on $G^r$ that multiplies  the $i$-th coordinate by $\gcd(B_i)$ for each $i\in [r]$, we have $M\circ P = P'\circ M^{(0)}$. It follows that for each $j\in [r]$ we have
\[
M_j (\Psi(g)) = M_j \circ P ( \Sigma \mathbf{x} )= \gcd(B_j) M_j^{(0)} (\Sigma \mathbf{x}) = \gcd(B_j)  \Big(M_j^{(0)}  {\mathbf{x}}^{(0)}+ \sum_{i=1}^r M_j^{(0)}  {\mathbf{x}'}^{(i)}\Big).
\]
Here we have firstly that $M_j^{(0)}  \mathbf{x}^{(0)}=0_G$, since $\mathbf{x}^{(0)}\in \ker_G M^{(0)}$ by definition of $\Phi$. For $i\in [r]\setminus\{j\}$, since $M_j^{(i)}$ restricted to $[m]$ equals $ M_j^{(0)}$, and the $(m+1)$-st component of $M_j^{(i)}$ is 0, we have $M_j^{(0)} {\mathbf{x}'}^{(i)}=M_j^{(i)} \mathbf{x}^{(i)}=0_G$. Finally, for $i=j$, we have $\gcd(B_j)\, M_j^{(0)} {\mathbf{x}'}^{(j)} =0_G$, since ${\mathbf{x}'}^{(j)}$ has coordinates in $G_j$. We have thus shown that  $M_j(\Psi(g))=0_G$ for each $j\in [r]$, hence $\Psi(g)\in \ker_G M$.

We now show that $\Psi(G_*^t)\supset \ker_G M$. Let $\mathbf{y}=(y_1,y_2,\ldots,y_m)\in \ker_G M$. For each $i\in [r]$, let $x_i= - B_i' \, (y_{r+1},\ldots,y_m)\in G$. Then the element
\[
\mathbf{x}^{(0)}:= (x_1,\ldots , x_r, y_{r+1},\ldots, y_m)
\]
satisfies $M^{(0)} \mathbf{x}^{(0)}=0$. Setting $\mathbf{x}=(\mathbf{x}^{(0)},0_{G^{m+1}},0_{G^{m+1}},\ldots, 0_{G^{m+1}})$, we have $K(\mathbf{x})=P\big(\mathbf{x}^{(0)}\big)=\mathbf{y}$. The map $\Phi$ is surjective (since each map $\Psi^{(i)}$ is), so there exists $g\in G_*^t$ such that $\Phi(g)=\mathbf{x}$, and so $\Psi(g)=K\circ \Phi(g)=\mathbf{y}$ as required.
\end{proof}

We now check condition (iii).
\begin{lemma}\label{lem:cond3final}
The map $\Psi$ in \eqref{eq:genrep} satisfies condition \textup{(iii)} from Definition \ref{defn:Cayley-rep}.
\end{lemma}
\begin{proof}
Fix any $j\in [m]$. From the definitions, it is clear that $\ker_{G_*} \psi_{C_j} \supset p_{C_j}(\ker_{G_*}\Psi)$. To prove the opposite containment, we shall use the fact that, by Lemma \ref{lem:equiv-def}, condition (iii') holds for each map $\Psi^{(i)}$.\\
\indent Recall the notations: $C_{j,i}\subset [t_i]$ is the support of the  $j$-th row of $\Psi^{(i)}$, and $C_j(i)\subset [t]$ for each $i$ is such that $C_j=\bigsqcup_{i=0}^r C_j(i)$ is the support of the $j$-th row of $\Psi$.\\
\indent Suppose that $g' \in G_*^{C_j}$ is given such that $\psi_{C_j}(g') = 0$. We want to find $g\in \ker_{G_*}\Psi$ such that $p_{C_j}(g) = g'$. We identify the groups $G_*^{C_{j,i}}$ and $G_*^{C_j(i)}$ the obvious way via the order-preserving bijection $C_{j,i}\to C_j(i)$.\\
\indent For each $i\in [0,r]$, let $g'^{(i)}\in G_*^{C_{j,i}}$ be the element such that we can identify $g'$ as  
\[
g'=\Big(g'^{(0)},\ldots,g'^{(r)}\Big)\in G_*^{C_{j,0}}\times G_*^{C_{j,1}}\times \cdots \times G_*^{C_{j,r}}.
\]
Our task is to show that there exists $g\in G_*^t$ such $\Psi(g)=0$ and $p_{C_j(i)}(g)= g'^{(i)}$ for each $i$ (modulo the above identification).

\underline{Case 1}: $j\in [r]$. 

For each $i\in [r]\setminus\{j\}$, let $\mathbf{x}^{(i)}$ be an element of $\ker_{G_i} M^{(i)}$ having the right $j$-th coordinate, i.e. such that the $j$-th row of $\Psi^{(i)}$, denoted $\psi_{C_{j,i}}^{(i)}$, satisfies $\psi_{C_{j,i}}^{(i)}(g'^{(i)})=\mathbf{x}^{(i)}_j$. Note that such a solution $\mathbf{x}^{(i)}$ can be obtained just by extending $g'^{(i)}$ arbitrarily to an element $g_0^{(i)}$ of $G_*^{t_i}$, for example by adding 0 coordinates; indeed we then have that $\Psi^{(i)}(g_0^{(i)})=\mathbf{x}^{(i)}$ lies in $\ker_{G_i} M^{(i)}$, as $\Psi^{(i)}$ satisfies condition (ii).\\
\indent Define 
\begin{equation}\label{eq:xdef1}
x_{\text{def}}:= -\psi_{C_{j,0}}^{(0)}\big(g'^{(0)}\big) -\sum_{i\in[r]\setminus\{j\}} \mathbf{x}^{(i)}_j = -\sum_{i\in[0,r]\setminus\{j\}}  \psi_{C_{j,i}}^{(i)}(g'^{(i)}).
\end{equation}
Note the important fact that $x_{\text{def}}\in G_j$. Indeed, letting $d_j=\gcd(B_j)$, and using that $d_j\, \psi_{C_{j,j}}^{(j)}(g'^{(j)})=0$ (since $d_j\, G_j=\{0_G\}$), we have
\[
d_j\; x_{\text{def}}  =  -\sum_{i\in[0,r]\setminus\{j\}} d_j\; \psi_{C_{j,i}}^{(i)}(g'^{(i)}) =  -\sum_{i\in[0,r]} d_j \;\psi_{C_{j,i}}^{(i)}(g'^{(i)}) = - \, \psi_{C_j}(g')=0.
\]
We can therefore find ${\mathbf{x}'}^{(j)}\in \ker_{G_j} M^{(0)}$ such that ${\mathbf{x}'}^{(j)}_j= -x_{\text{def}}$ (using the fact that the rows of $B'$ in $M^{(0)}$ are coprime).\\ 
\indent Define $\mathbf{x}^{(j)}$ to be the element of $ \ker_{G_j} M^{(j)}$ that restricts to ${\mathbf{x}'}^{(j)}$ on its first $m$ coordinates. (Thus $\mathbf{x}^{(j)}$ is ${\mathbf{x}'}^{(j)}$ with an extra $(m+1)$-st coordinate equal to $x_{\text{def}}$.)\\
\indent For each $i\in[r]\setminus\{j\}$, since $\Psi^{(i)}$ is a representation, by condition (iii') from Definition \ref{defn:Cayley-rep-var} we can extend $g'^{(i)}$ to some $g^{(i)}\in G_i^{t_i}$ such that $\Psi^{(i)}(g^{(i)})= \mathbf{x}^{(i)}$.\\
\indent Now we want to extend $g'^{(j)}$ to an element $g^{(j)}$ such that $\Psi^{(j)}(g^{(j)})$ agrees with $\mathbf{x}^{(j)}$ at each of its first $m$ coordinates except perhaps the $j$-th one; equivalently, we want $(\Psi^{(j)}(g^{(j)}))_u= \mathbf{x}^{(j)}_u$ for each $u\in[r+1,m]$. (This implies equality also for $u\in [r]\setminus \{j\}$ since $\Psi^{(j)}(g^{(j)})$ and $\mathbf{x}^{(j)}$ are both in $\ker_{G_j} M^{(j)}$.)\\ 
\indent We can find this element $g^{(j)}$ thanks to the freedom in the $(m+1)$-st variable in the system $(M^{(j)},G_j)$. In other words, we are using condition (iii') for $\Psi^{(j)}$ to extend $g'^{(j)}$, but we are doing so with target-solution the element of $\ker_{G_j}M^{(j)}$ that has $j$-th coordinate $\psi^{(j)}_{C_{j,j}}(g'^{(j)})$ and $u$th coordinate $\mathbf{x}^{(j)}_u$ for $u\in [m]\setminus \{j\}$, and we are using the freedom in the $(m+1)$-st coordinate to claim that such a target-solution exists.\\
\indent We finally come to extending $g'^{(0)}$, and to do so we first have to choose an appropriate element $\mathbf{x}^{(0)}\in \ker_{G_0} M^{(0)}$.  Consider the element $\mathbf{x}^{(0)}$ such that the restriction $\mathbf{x}^{(0)}|_{[r+1,m]}$ of $\mathbf{x}^{(0)}$ to coordinates indexed in $[r+1,m]$ satisfies
\begin{equation}\label{eq:def_sol}
\mathbf{x}^{(0)}|_{[r+1,m]}=-\sum_{i\in[r]} \mathbf{x}^{(i)}|_{[r+1,m]},
\end{equation}
for the $\mathbf{x}^{(i)}$ defined above.\\
\indent The key claim now is that this solution $\mathbf{x}^{(0)}\in \ker_{G_0} M^{(0)}$, determined by \eqref{eq:def_sol}, satisfies $\mathbf{x}^{(0)}_j=\psi^{(0)}_{C_{j,0}}(g'^{(0)})$. If this holds then we may use condition (iii') to obtain the desired extension $g^{(0)}$ of $g'^{(0)}$ such that $\Psi^{(0)}(g^{(0)})=\mathbf{x}^{(0)}$.\\
\indent To prove the claim, note that on one hand by \eqref{eq:xdef1} we have
\begin{equation}\label{eq:fin1}
\psi^{(0)}_{C_{j,0}}(g'^{(0)})  = -  x_{\text{def}} - \sum_{i\in [r]\setminus \{j\}} \mathbf{x}_j^{(i)}.
\end{equation}
On the other hand, letting $B_j^{(0)}$ denote the restriction of the row $M_j^{(0)}$ to the entries indexed by $[r+1,m]$, we deduce from $M_j^{(0)}(\mathbf{x}^{(0)})=0$ that
\[
\mathbf{x}^{(0)}_j =  -B_j^{(0)} \mathbf{x}^{(0)}|_{[r+1,m]} = \sum_{i\in[r]}B_j^{(0)} \mathbf{x}^{(i)}|_{[r+1,m]}=   \sum_{i\in [r]} B_j^{(0)} \big(\mathbf{x}^{(i)}_{r+1},\ldots,\mathbf{x}^{(i)}_m\big).
\]
Here the summand with index $i=j$ is $\mathbf{x}^{(j)}_j=-x_{\text{def}}$, and for each $i\in [r]\setminus \{j\}$ the $i$-th summand is $B_j^{(i)}\big(\mathbf{x}^{(i)}_{r+1},\ldots,\mathbf{x}^{(i)}_m\big)=- \mathbf{x}^{(i)}_j$, since $M^{(i)}$ has same $j$-th row as $M^{(0)}$.
Hence
\begin{equation}\label{eq:fin2}
\mathbf{x}^{(0)}_j  = \;-x_{\text{def}} \; -\; \sum_{i\in [r]\setminus \{j\}} \mathbf{x}^{(i)}_j.
\end{equation}
Combining \eqref{eq:fin1} and \eqref{eq:fin2} we deduce that $\psi^{(0)}_{C_{j,0}}(g'^{(0)}) = \mathbf{x}^{(0)}_j$ as claimed. 

We have thus obtained $g=\big(g^{(0)},g^{(1)},\ldots, g^{(r)}\big)$ such that $p_{C_j}(g)= g'$, and such that for each $u\in [r+1,m]$ we have
$\Psi_u(g)=\sum_{i\in [0,r]} \Psi_u^{(i)}\big(g^{(i)}\big) = \sum_{i\in [0,r]} \mathbf{x}_u^{(i)} = 0$, 
by definition of $\mathbf{x}^{(0)}$. Since an element of $\ker_G M$ is determined by its last $m-r$ coordinates, we must have $\Psi(g)=0$. This completes Case 1.

\underline{Case 2}: $j\in[r+1,m]$.

The argument is similar but simpler. Suppose that we are given 
\[
g'=\big({g'}^{(0)},\ldots, {g'}^{(r)}\big)\in G_*^{C_{j,0}}\times\cdots\times G_*^{C_{j,r}}
\]
such that $\sum_{i\in [0,r]}\psi_{C_{j,i}}^{(i)}\big({g'}^{(i)}\big)=0$. For each $i\in [r]$, let $g^{(i)}\in G_*^{t_i}$ be any element satisfying $p_{C_{j,i}}(g^{(i)})={g'}^{(i)}$ (e.g. obtained by extending ${g'}^{(i)}$ by $0$-coordinates) and let $\mathbf{x}^{(i)}=\Psi^{(i)}\big(g^{(i)}\big)\in \ker_{G_i} M^{(i)}$. (Note that $\psi_{C_{j,i}}^{(i)}\big({g'}^{(i)}\big)=\mathbf{x}_j^{(i)}$.) For each $i\in [r]$, we define
\[
\mathbf{y}^{(i)}= \Big(\,-\mathbf{x}^{(i)}_1,\ldots,-\mathbf{x}^{(i)}_{i-1},B_i'\big(\mathbf{x}^{(i)}_{r+1},\ldots, \mathbf{x}^{(i)}_m \big) ,-\mathbf{x}^{(i)}_{i+1},\ldots, -\mathbf{x}^{(i)}_m\, \Big).
\]
Note that $\mathbf{y}^{(i)}\in \ker_G M^{(0)}$. Therefore, the element $\mathbf{x}^{(0)}:= \sum_{i\in [r]} \mathbf{y}^{(i)}$ lies in $\ker_G M^{(0)}$. We also have \[
\mathbf{x}^{(0)}_j = - \sum_{i\in [r]} \mathbf{x}^{(i)}_j= -\sum_{i\in [r]} \psi^{(i)}_{C_{j,i}}({g'}^{(i)})=\psi^{(0)}_{C_{j,0}}({g'}^{(0)}),
\]
where the first equality follows from the definition of the $\mathbf{y}^{(i)}$, the second from the definition of the $\mathbf{x}^{(i)}$, and the third equality follows by assumption on $g'$. By condition (iii'),  there exists $g^{(0)}$ extending ${g'}^{(0)}$ such that $\Psi^{(0)}(g^{(0)})=\mathbf{x}^{(0)}$. The element $g=(g^{(0)},\ldots, g^{(r)})$ that we have thus obtained extends $g'$, and for each $u\in [r+1,m]$ we have $\Psi_u(g)=\sum_{i\in [0,r]} \Psi_u^{(i)}\big(g^{(i)}\big) = \sum_{i\in [0,r]} \mathbf{x}_u^{(i)} = 0$, so just like in the previous case we must have $\Psi(g)=0$. This completes Case 2.
\end{proof}

The proof of Proposition \ref{prop:genrep} is now complete.\\

\noindent It remains only to address the second case described at the beginning of this subsection, namely that of matrices of the form $(I_r|B)\in \Z^{r\times (r+1)}$. We shall do so by using a matrix extension  slightly different from those used in earlier sections. This replaces the given matrix by an $(r+1)\times (r+3)$ matrix. We show that this extension also conserves hypergraph-representability, so that we can then just apply Proposition \ref{prop:genrep} to the latter matrix.

\begin{lemma}
Let $M=(I_r|B)\in \Z^{r\times (r+1)}$. Let
\[
M'=\left(\begin{array}{c|c}
\begin{array}{c}
	I_{r+1}
	\end{array}         & 
	\begin{array}{cc}
		B  & \mathbf{0}\\
          0  &   -1      \\
\end{array}
\end{array}\right)\in \Z^{(r+1)\times (r+3)}.
\]
Suppose that $\Psi'$ is a $(t,r+3,k)$-representation for $M'$, and let $J=[r+2]\setminus \{r+1\}$. Then $\Psi:= \pi_J\circ \Psi'$ is a $(t,r+1,k)$-representation for $M$.
\end{lemma}
\begin{proof}
Condition (i) from Definition \ref{defn:Cayley-rep} is clearly inherited  by $\Psi$ from $\Psi'$.\\
\indent For condition (ii), note that $\Psi'(G_*^t)=\ker_G M'$ by assumption. Thus for any $x'\in \Psi'(G_*^t)$ we have $M' x'=0_{G^{r+3}}$, which implies by construction that $x:=\pi_J(x')$ satisfies $Mx=0_{G^{r+1}}$. Hence $\Psi(G_*^t)\subset \ker_G M$. To see the opposite containment, let $x\in \ker_G M$, and note that $x':=(x_1,x_2,\ldots,x_r,0,x_{r+1},0)$ lies in $\ker_G M'$, so there exists $g\in G_*^t$ such that $\Psi'(g)=x'$ and so $\Psi(g)=\pi_J(x')=x$.\\
\indent To check condition (iii), fix any $j\in [r+1]$ and note that by  assumption we have $p_{C'_{\sigma_J(j)}}(\ker_{G_*}\Psi') =\ker_{G_*}\psi'_{C'_{\sigma_J(j)}}$. By construction, the $j$-th row of $\Psi$ has support $C_j=C'_{\sigma_J(j)}$, where the latter is the support of the $\sigma_J(j)$th row of $\Psi'$. We also have that the corresponding maps $\psi_{C_j}, \psi'_{C'_{\sigma_J(j)}}$ are equal, whence $p_{C_j}(\ker_{G_*}\Psi') =\ker_{G_*}\psi_{C_j}$. Therefore it suffices to check that $p_{C_j}(\ker_{G_*}\Psi')=p_{C_j}(\ker_{G_*}\Psi)$.\\ 
\indent The rows of $\Psi$ form a subset of those of $\Psi'$, so we certainly have $\ker_{G_*}\Psi'\subset \ker_{G_*}\Psi$ and so $p_{C_j}(\ker_{G_*}\Psi')\subset p_{C_j}(\ker_{G_*}\Psi)$. For the opposite containment, suppose that $g'\in G_*^{C_j}$ equals $p_{C_j}(g_0)$ for some $g_0\in \ker_{G_*} \Psi$, so in particular $\psi_{C_j}(g')=0$. Applying (iii') from Definition \ref{defn:Cayley-rep-var} to $\Psi'$, with $x=0_{G^{r+3}}$ and $g'$ satisfying $\psi_{C'_{\sigma(j)}}(g')=x_{\sigma(j)}=0$, we obtain that there exists $g\in G_*^t$ such that $\Psi'(g)=0_{G^{r+3}}$ and $p_{C_j}(g)=p_{C'_{\sigma(j)}}(g)=g'$.
\end{proof}

\section{Remarks}\label{section:Remarks}

There are several ways in which one could try to extend Theorem \ref{thm:cag-rem-lem} further.

To begin with, one may want to remove the assumption $d_r(M)=1$. To achieve this, the arguments in this paper would have to be modified in a non-trivial way, especially those in Section \ref{section:FindRep}, starting with Lemma \ref{lem:firstext}, and including the proofs of Lemmas \ref{lem:cond2final} and \ref{lem:cond3final}.\\
\indent One may also want to bring Theorem \ref{thm:cag-rem-lem} more in line with the $\Z_p$ version (Theorem \ref{thm:KSVremlem}) by making sure that the parameter $\delta$ depends only on the dimensions $m,r$ of the matrix $M$ and not on the entries themselves. Note that the function $\delta$ in Theorem \ref{thm:cag-rem-lem} is currently not guaranteed to be independent of the entries of $M$, because of the argument involving Euclid's algorithm in the proof of Lemma \ref{lem:adapt11}. One would therefore need at least to modify Lemma \ref{lem:adapt11}.\\
\indent Thus, obtaining the above two improvements of Theorem \ref{thm:cag-rem-lem} via our approach requires handling  several technical difficulties of a purely algebraic nature, and we have therefore preferred not to pursue these matters in this paper.

One may also seek extensions of these removal results to noncommutative settings. It seems plausible, for instance, that there is an analogue of Theorem \ref{thm:cag-rem-lem} for nilmanifolds. One possible such result would say, roughly speaking, that if a product of measurable subsets of a nilmanifold $G/\Gamma$ has an intersection of small-measure with the so-called `Leibman nilmanifold' associated with a system of linear forms (see \cite[\S 3]{G&T}), then this intersection can be eliminated by removing small-measure subsets from the given sets. 

\begin{appendix}

\section{Reduction of the main theorem} \label{app:A}

In this appendix we show that Theorem \ref{thm:cag-rem-lem+} implies Theorem \ref{thm:cag-rem-lem}. In fact, as we shall see, it is not hard to establish the following stronger reduction.

\begin{proposition}\label{prop:3.1_implies_1.3}
If Theorem \ref{thm:cag-rem-lem+} holds for every compact abelian Lie group, then Theorem \ref{thm:cag-rem-lem} holds.
\end{proposition}

As is well-known, every compact Hausdorff abelian group $G$ is a strict projective limit of compact abelian Lie groups (see \cite[Rem. 2.35, Cor. 2.43]{H&M}). We shall use this to prove Proposition \ref{prop:3.1_implies_1.3}, by approximating the given Borel sets $A_j\subset G$ in Theorem \ref{thm:cag-rem-lem} by Borel subsets coming from a Lie quotient of $G$. More precisely, we use the following approximation result.

\begin{lemma}\label{lem:Cagapprox}
Let $G$ be a compact abelian group, let $A$ be a Borel  subset of $G$, and let $0<\delta < 1$. There exists a compact abelian Lie group $G_0$, a continuous surjective homomorphism $q: G\to G_0$, and a Borel set $A_0\subset G_0$, such that $\mu_G\big(A\, \Delta\, (q^{-1}A_0) \big) < \delta$.
\end{lemma}
\begin{proof}
By Lusin's theorem there exists a continuous function $h$ on $G$ with $\|h\|_{L^\infty(G)}\leq 1$ such that $\|h-1_A\|_{L^1(G)}< \delta^3/2^{10}$; see \cite[Appendix E8]{Rud2}. By the Stone-Weierstrass theorem, the trigonometric polynomials are dense in the set of continuous functions on $G$, relative to the $L^\infty(G)$-norm; see \cite[p. 24]{Rud2}. Thus there exists a trigonometric polynomial $P(x)$ such that $\|h-P\|_{L^\infty(G)}<\delta^3/2^{10}$, whence $\| 1_A-P\|_{L^1(G)} < \delta^3/2^9$. We also have $\|P\|_{L^\infty(G)}< \|h\|_{L^\infty(G)}+ \delta^3/2^{10}< 2$, and by taking real parts we can also suppose that $P$ is real-valued.\\
\indent Let $\widehat G$ be the dual group of $G$ and let $\widehat{G_0}$ be the subgroup of $\widehat G$ generated by the spectrum of $P$, i.e. by the finite set $\{\gamma\in\widehat G : \widehat P(\gamma)\neq 0\}$. Then $\widehat G_0$ is a finitely generated (discrete) abelian group, and is thus the dual of a compact abelian Lie group $G_0$. Letting $\Lambda$ denote the annihilator of $\widehat G_0$ ($\Lambda$ is a closed subgroup of $G$), we have that $G_0$ is isomorphic as a compact abelian group to $G/\Lambda$ (see \cite[\S 2.1]{Rud2}),  and so the quotient map $G\to G/\Lambda$ gives a continuous surjective homomorphism $q :G\to G_0$. There exists a trigonometric polynomial $P_0$ on $G_0$ with $P= P_0\circ q$. We then have $\|P_0\|_{L^\infty(G_0)}\leq 2$. Moreover,
\[
\Big\|P_0-P_0^2\Big\|_{L^1(G_0)} = \Big\|P-P^2\Big\|_{L^1(G)}\leq \int_G |1_A-P|\ud\mu_G+ \int_G |1_A-P|\,|1_A+P| \ud \mu_G < \delta^3/2^7.
\]
This implies that the set $D=\{x\in G_0:|P_0(x)-P_0^2(x)|  >\delta^2/2^4 \}$ has measure at most $\delta/8$. On the complement $D^c=G\setminus D$, we must have $|P_0(x)|\leq \delta/4$ or $|1-P_0(x)|\leq \delta/4$. Now let $A_0=\{x\in G_0: |P_0(x)-1|\leq \delta/4\}$. We have that $\|1_{A_0}-P_0 \|_{L^1(G_0)} $ is at most
\[
3 \int_{G_0} 1_{D}\ud\mu_{G_0} +  \int_{G_0} 1_{A_0\cap D^c}(x)|1-P_0(x)| +1_{A_0^c\cap D^c}(x) |P_0(x)| \ud\mu_{G_0} < 7\delta/8.
\]
Hence $\mu_G\big(A\,\Delta\, (q^{-1}A_0)\big) \leq \| 1_A-P \|_{L^1(G)} + \| P\,- \,1_{A_0}\circ q \,\|_{L^1(G)} 
= \| 1_A-P\|_{L^1(G)} + \| P_0-1_{A_0} \|_{L^1(G_0)}< \delta$.
\end{proof}
By iterating the main argument in this proof we can simultaneously approximate any finite number of Borel sets $A_1,A_2,\ldots, A_m\subset G$, that is we can find a single Lie group $G_0$ in which there are Borel sets $A_{j,0}$ such that $\mu_G\big(A_j\, \Delta\, (q^{-1}A_{j,0}) \big) < \delta$ for every $j\in [m]$.\\
\indent We shall also use the following basic fact.

\begin{lemma}\label{lem:basic1}
Let $M\in \Z^{r\times m}$ satisfy $d_r(M)=1$. Then for any abelian group $G$, and any surjective homomorphism $\theta:G\to H$, the homomorphism $\theta^m:G^m\to H^m, x\mapsto (\theta(x_1),\ldots, \theta(x_m))$ is surjective from $\ker_G M$ to $\ker_H M$.
\end{lemma}
\begin{proof}
The homomorphism $M:G^m\to G^r$ is surjective for any abelian group $G$; this is immediate from the Smith normal form $M= U\, (I_r| 0^{r\times (m-r)})\, V$, where $U\in \Z^{r\times r}, V\in \Z^{m\times m}$ are unimodular matrices. On $\ker_G M$, the homomorphism $\theta^m$ takes values in $\ker_H M$. Now given $x_H \in \ker_H M$, there exists  $x\in G^m$ such that $\theta^m(x)= x_H$. We have $0=M(x_H)=M(\theta^m(x))=\theta^r(M(x))$, so $M(x)\in (\ker \theta)^r\leq G^r$. By surjectivity of $M$, there exists $z\in (\ker \theta)^m$ such that $M(z)=M(x)$. Thus $x-z$ is an element of $\ker_G M$ satisfying $\theta^m(x-z)= x_H$, so $\theta^m$ is indeed onto $\ker_H M$.
\end{proof}

Finally, we shall also use the fact that the integral  of bounded functions across a kernel $\ker_G M$ can be controlled in terms of their $L^1(G)$ norms, in the following sense.

\begin{lemma}\label{lem:basic2}
Let $M\in \Z^{r\times m}$, and let $G$ be a compact abelian group. For $j\in [m]$ let $p_j: G^m\to G$ denote the projection homomorphism to the $j$-th component, and let $G_j$ denote the closed subgroup $p_j(\ker_G M)\leq G$. Suppose that each $G_j$ has finite index $\kappa_j=|G:G_j |$ in $G$. Then, for any measurable functions  $f_1,\ldots,f_m:G\to \C$ with $\|f_j\|_{L^\infty(G)}\leq 1$ for all $j$, we have
\begin{equation}\label{eq:basic2}
\Big| \int_{\ker_G M}\; f_1(x_1)\cdots f_m(x_m) \ud \mu_{\ker_G M} (x) \Big| \leq \min_{j\in [m]}\;\kappa_j \;  \norm{f_j 1_{G_j}}_{L_1(G)}.
\end{equation}
\end{lemma}

\begin{proof}
Fix $j\in [m]$. By the triangle inequality and the bounds $\|f_j\|_{L^\infty(G)}\leq 1$, the left side of \eqref{eq:basic2} is at most
\[
 \int_{\ker_G M}\; |f_j(x_j)| \ud \mu_{\ker_G M} (x) = \int_{\ker_G M}\; |f_j\circ p_j(x)| \ud \mu_{\ker_G M} (x)
 =  \int_{G_j}\; |f_j(y)| \ud \mu_{G_j} (y).
\]
On the other hand, by the quotient integral formula \cite[Theorem 1.5.2]{D&E}, we have
\begin{eqnarray*}
\int_G | f_j(x) 1_{G_j}(x) | \ud\mu_G(x) & = &\kappa_j^{-1} \sum_{z\in G/G_j} \int_{G_j} |f_j (y+z)| 1_{G_j}(y+z) \ud\mu_{G_j}(y)\\
& = &\kappa_j^{-1} \int_{G_j}\; |f_j(y)| \ud \mu_{G_j} (y).
\end{eqnarray*}
The result follows.
\end{proof}

\begin{proof}[Proof of Proposition \ref{prop:3.1_implies_1.3}]
Given $M\in \Z^{r\times m}$ with $d_r(M)=1$, let $\delta'>0$ be such that Theorem \ref{thm:cag-rem-lem+} holds for any compact abelian Lie group with initial parameter $\epsilon/2$, and let $\delta=\min(\delta'/2,\epsilon/2)$. Suppose that $A_1,\ldots, A_m$ are Borel subsets of a compact abelian group $G$ satisfying
$\mu_{\ker_G M}(A_1\times \cdots \times A_m \cap \ker_G M)\leq \delta$. Note that we may assume without loss of generality that each closed subgroup $G_j=p_j(\ker_G M)$ has positive measure in $G$, for otherwise the conclusion of Theorem \ref{thm:cag-rem-lem} holds just by removing the null-set $G_j\cap A_j$. It follows that each $G_j$ has finite index $\kappa_j$ in $G$. Indeed, by \cite[Corollary 20.17]{H&R}, there is a non-empty open set $U$ contained in the difference set $G_j-G_j = G_j$, and by compactness $G$ can be covered by finitely many translates of $U$ and therefore of $G_j$.\\
\indent It follows from Lemma \ref{lem:Cagapprox} (or rather a repeated application of its proof) that there exists a compact abelian Lie group $G_0$, a continuous surjective homomorphism $q:G\to G_0$, and a Borel set  $A_{j,0}\subset G_0$ for each $j\in [m]$, such that $\mu_G\big(A_j\, \Delta\, (q^{-1} A_{j,0})\big)\leq \delta/\kappa_j m$ for every $j\in [m]$. We claim that we  therefore have
\[
\mu_{\ker_{G_0} M}(A_{1,0}\times \cdots \times A_{m,0}  \cap \ker_{G_0} M)\leq \delta'.
\]
Indeed, firstly by Lemma \ref{lem:basic1} the homomorphism $q^m: \ker_G M\to \ker_{G_0} M$ is surjective, whence
\[
\mu_{\ker_{G_0} M}\big(A_{1,0}\times \cdots \times A_{m,0} \cap \ker_{G_0} M\big)= \mu_{\ker_G M}\big(q^{-1}A_{1,0}\times \cdots \times q^{-1}A_{m,0} \cap \ker_G M\big).
\]
Now, the map $(f_1,\ldots, f_m)\mapsto \int_{\ker_G M} f_1(x_1)\cdots f_m(x_m) \ud\mu_{\ker_G M}(x)$ is multilinear (for measurable functions $f_i$), and reduces to $\mu_{\ker_G M}(A_1\times\cdots\times A_m\cap \ker_G M)$ when $f_i=1_{A_i}$. A multilinearity argument (similar to that in the proof of Lemma \ref{lem:measrem}) gives us
\[
 \mu_{\ker_G M}\big(\prod_{j\in [m]}q^{-1}A_{j,0}\, \cap\, \ker_G M\big)= \sum_{j\in [m+1]}  \int_{\ker_G M} f_{1,j}(x_1)\cdots f_{m,j}(x_m) \ud\mu_{\ker_G M}(x),
\]
where for each $j\in [m+1]$ we have $f_{i,j} = 1_{A_i}$ if $i<j$, we have $f_{j,j} = 1_{q^{-1}A_{j,0}}-1_{A_j}$, and $f_{i,j} = 1_{q^{-1}A_{i,0}}$ if $i>j$. By Lemma \ref{lem:basic2}, it follows that $\mu_{\ker_G M}(q^{-1}A_{1,0}\times \cdots \times q^{-1}A_{m,0} \cap \ker_G M)$ is at most 
\[
 \mu_{\ker_G M}(A_1\times \cdots \times A_m \cap \ker_G M)+ \sum_{j\in [m]} \,\kappa_j\, \mu_G\big(A_j\,\Delta\, (q^{-1} A_{j,0})\big)  \leq \delta',
 \]
as we claimed.\\
\indent We now apply Theorem \ref{thm:cag-rem-lem+} on $G_0$, obtaining sets $R_{j,0}$ of measure at most $\epsilon/2$ such that $\prod_{j\in [m]} A_{j,0}\setminus R_{j,0}$ is $M$-free. Then $\prod_{j\in [m]} q^{-1} (A_{j,0})\setminus q^{-1}(R_{j,0})$ is $M$-free, whence, setting $R_j=q^{-1} (R_{j,0}) \cup (A_j \setminus q^{-1} (A_{j,0}))$, we are done.
\end{proof}

\begin{remark}\label{rem:TZ}
One can obtain a version of Lemma \ref{lem:Cagapprox} in which the approximating group $G_0$ is just second countable, arguing along the following lines. By Plancherel's theorem, the Fourier transform of $1_A$ is square-summable, hence supported on a countable subset of the Pontryagin dual $\widehat G$, hence supported on a countable subgroup of $\widehat G$. Taking duals then yields an approximation of $A$, up to a null set, given by a subset $A_0$ of a second countable quotient of $G$. 
\end{remark}

\end{appendix}

\noindent \textbf{Acknowledgements.} The authors are very grateful to the following institutions for their support. The first named author was supported by the \'Ecole normale sup\'erieure, Paris, and the Fondation Sciences Math\'ematiques de Paris, and his work was part of project ANR-12-BS01-0011 CAESAR. The second named author is supported by the ERC Consolidator Grant No. 617747. The third named author is supported by a University of Toronto Graduate Student Fellowship. The authors are also grateful to Bernard Host for indicating the proof of Lemma \ref{lem:Cagapprox}, to Terence Tao and Tamar Ziegler for Remark \ref{rem:TZ}, and to anonymous referees for useful detailed comments.

\end{document}